\documentclass[11pt]{article}
\usepackage{amssymb}
\usepackage{amsmath}
\usepackage{latexsym}
\usepackage{mathrsfs}
\usepackage{verbatim}
\usepackage{multirow}
\usepackage{algorithm}
\usepackage{algorithmic}

\newtheorem{thm}{Theorem}
\newtheorem{coro}{Corollary}

\newtheorem{prop}{Proposition}

\newenvironment{proof}{{\bf Proof\,\,}}{\endproof\par}

\newcounter{spb}
\newcommand{\subpb}{(\alph{spb}) \addtocounter{spb}{1}}

\newcommand{\resetspb}{\setcounter{spb}{1}}

\def \openbox{$\sqcup\llap{$\sqcap$}$}
\def \endproof{\enskip \null \nobreak \hfill \openbox \par}

\newcommand{\limto}{\rightarrow}

\newcommand{\bz}{\mathbf 0}

\newcommand{\R}{\mathbb R}
\newcommand{\Q}{\mathbb Q}

\newcommand{\E}[1]{{\mathbb E}\left[ #1 \right]}

\begin{document}
\title{Hardness and Approximation Results for $L_p$--Ball Constrained Homogeneous Polynomial Optimization Problems\thanks{This research was supported by the Hong Kong Research Grants Council (RGC) General Research Fund (GRF) Project CUHK 419409.}}
\author{Ke Hou\thanks{Department of Systems Engineering and Engineering Management, The Chinese University of Hong Kong, Shatin, N.~T., Hong Kong.  E--mail: {\tt khou@se.cuhk.edu.hk}} \and Anthony Man--Cho So\thanks{Department of Systems Engineering and Engineering Management, and by courtesy, Department of Computer Science and Engineering and the CUHK--BGI Innovation Institute of Trans--omics, The Chinese University of Hong Kong, Shatin, N.~T., Hong Kong.  E--mail: {\tt manchoso@se.cuhk.edu.hk}}}
\date{\today}
\maketitle

\begin{abstract}
In this paper, we establish hardness and approximation results for various $L_p$--ball constrained homogeneous polynomial optimization problems, where $p \in [2,\infty]$.  Specifically, we prove that for any given $d \ge 3$ and $p \in [2,\infty]$, both the problem of optimizing a degree--$d$ homogeneous polynomial over the $L_p$--ball and the problem of optimizing a degree--$d$ multilinear form (regardless of its super--symmetry) over $L_p$--balls are NP--hard.  On the other hand, we show that these problems can be approximated to within a factor of $\Omega\left( (\log n)^{(d-2)/p} \big/ n^{d/2-1} \right)$ in deterministic polynomial time, where $n$ is the number of variables.  We further show that with the help of randomization, the approximation guarantee can be improved to $\Omega((\log n/n)^{d/2-1})$, which is independent of $p$ and is currently the best for the aforementioned problems.  Our results unify and generalize those in the literature, which focus either on the quadratic case or the case where $p \in \{2,\infty\}$.  We believe that the wide array of tools used in this paper will have further applications in the study of polynomial optimization problems.

\medskip
\noindent {\bf Keywords:} Polynomial Optimization; Approximation Algorithms; Diameters of Convex Bodies; Convex Programming

\medskip
\noindent {\bf Mathematics Subject Classification:} 15A69, 90C26, 90C59
\end{abstract}

\section{Introduction}
Motivated by its diverse applications and profound connections to various branches of mathematics, polynomial optimization has been the focus of much research effort during the past decade or so.  From an algorithmic perspective, polynomial optimization problems are generally intractable.  Thus, a fundamental research issue is to determine their approximability.  One important class of problems whose approximability has been extensively investigated in recent years is that of homogeneous polynomial optimization with $L_2$--norm constraints.  The first results in this direction were obtained by de Klerk et al.~\cite{dKLP06} and Barvinok~\cite{Barvinok07}, who showed that certain specially structured $L_2$--sphere constrained polynomial optimization problems admit polynomial--time approximation schemes (PTASes).  These were then followed by the work of Luo and Zhang~\cite{LZ10}, in which an approximation algorithm was developed for homogeneous quartic optimization problems with quadratic constraints (which includes the $L_2$--ball as a special case).  Around the same time, Ling et al.~\cite{LNQY09} considered the problem of approximately optimizing a biquadratic function over the Cartesian product of two $L_2$--spheres; while Zhang et al.~\cite{ZQY12} studied the hardness and approximability of certain $L_2$--sphere constrained homogeneous cubic optimization problems.  Since then, there have been significant activities in this line of research.  For instance, in~\cite{ZLQ11,LZQ12,YY12}, various researchers derived approximation results for the problem of optimizing a biquadratic function over quadratic constraints, thereby extending the results in~\cite{LNQY09}.  In~\cite{HLZ10}, He et al.~improved and substantially extended the results in~\cite{LZ10} by providing approximation algorithms for optimizing a general homogeneous polynomial over quadratic constraints (see also~\cite{LHZ12} for some latest developments).  It is worth noting that most of the aforementioned results were obtained using semidefinite relaxation techniques, and that most of the algorithms are randomized.  Recently, in a marked departure from the semidefinite relaxation paradigm, So~\cite{S11a} employed techniques from algorithmic convex geometry to design deterministic approximation algorithms for various $L_2$--sphere constrained homogeneous polynomial optimization problems.  The algorithms in~\cite{S11a} have a worst--case approximation guarantee of $\Omega((\log n/n)^{d/2-1})$, where $n$ is the number of variables and $d$ is the degree of the polynomial.  Roughly speaking, this means that given any problem instance, the algorithms will produce a feasible solution whose objective value is at least $\Omega((\log n/n)^{d/2-1})$ times the optimum.  This improves upon the $\Omega((1/n)^{d/2-1})$ bound established in~\cite{LNQY09,HLZ10,ZQY12} and is currently the best for general $L_2$--sphere constrained homogeneous polynomial and multiquadratic optimization problems.  Such development raises a natural question: Can the approach in~\cite{S11a} be applied to other classes of polynomial optimization problems?


In this paper, we address the above question by extending the approach in~\cite{S11a} to study the $L_p$--ball constrained homogeneous polynomial optimization problem; i.e., problem of the form
\begin{equation} \label{eq:generic-lp-opt}
   \max\{ f(x): \|x\|_p \le 1 \},
\end{equation}
where $p \in [2,\infty]$ and $f:\R^n\limto\R$ is a homogeneous polynomial of (fixed) degree $d\ge3$.  Our motivation for studying Problem (\ref{eq:generic-lp-opt}) is twofold.  First, it is a natural extension of the matrix norm problem in~\cite{BN01b,Steinberg05} and the $L_p$--Grothendieck problem in~\cite{KNS10}---both of which concern quadratic $f$'s with certain structure---as well as the $L_\infty$--ball constrained trilinear optimization problem in~\cite{KN08} and the $L_2$--ball constrained homogeneous polynomial optimization problem in~\cite{LZ10,HLZ10}.  However, to the best of our knowledge, there is no prior hardness or approximation result for Problem (\ref{eq:generic-lp-opt}) in its full generality.  Secondly, Problem (\ref{eq:generic-lp-opt}) lies at the heart of many applications.  For instance, Baratchart et al.~\cite{BBP98} demonstrated that many labeling problems in pattern recognition and image processing can be tackled by maximizing a certain polynomial over an $L_p$--ball.  In addition, the $L_p$--singular value and singular vector of a tensor, which have been extensively studied in the spectral theory of tensors and play an important role in signal processing, automatic control and data analysis, can be defined as the optimal value of and optimal solution to an $L_p$--ball constrained homogeneous polynomial optimization problem, respectively~\cite{L05,Q05}.  As our main contribution, we obtain both hardness and approximation results for Problem (\ref{eq:generic-lp-opt}).  Specifically, on the hardness side, we show that Problem (\ref{eq:generic-lp-opt}) is NP--hard for any given $d\ge3$ and $p\in[2,\infty]$.  To the best of our knowledge, this is the first hardness result for Problem (\ref{eq:generic-lp-opt}) that holds for any given $d\ge3$ and $p\in[2,\infty]$.  By contrast, existing hardness results for Problem (\ref{eq:generic-lp-opt}), such as those in~\cite{N03,HL09,AOPT11,ZQY12}, hold only for certain values of $d$ and $p$.  A key tool we used to prove the hardness result is a tensor symmetrization procedure introduced by Ragnarsson and Van Loan~\cite{RvL11}, which allows us to establish the equivalence between multilinear optimization problems and certain homogeneous polynomial optimization problems.  On the approximation side, we show that Problem (\ref{eq:generic-lp-opt}) can be approximated to within a factor of $\Omega\left( (\log n)^{(d-2)/p} \big/ n^{d/2-1} \right)$ by a deterministic polynomial--time algorithm.  Furthermore, if one allows randomization, then the approximation bound can be improved to $\Omega((\log n/n)^{d/2-1})$, independent of $p$.  In the process of deriving these results, we also establish the hardness of and develop approximation algorithms for certain $L_p$--ball constrained multilinear optimization problems, which could be of independent interest.  We remark that the aforementioned results apply only to the case where $p\in[2,\infty]$.  The case where $p\in[1,2)$, which is not covered in this paper, does not seem to be well understood, even when $f$ is quadratic.  We refer the interested reader to~\cite{Steinberg05,dKLP06,BV11} for some results in this direction.

Before describing in detail our approximation algorithms for Problem (\ref{eq:generic-lp-opt}), let us give an overview of our approach and highlight some of the key technical issues.  To fix ideas, let us first consider the case where $d=3$; i.e., $f(x) = \sum_{i,j,k=1}^n a_{ijk}x_ix_jx_k$ for some order--$3$ tensor $\mathcal{A}=(a_{ijk}) \in \R^{n\times n\times n}$.  Using by--now standard techniques (see, e.g.,~\cite{HLZ10,S11a}), one can show that the optimal value of Problem (\ref{eq:generic-lp-opt}) is within a constant factor of that of its multilinear relaxation, which in the case of $d=3$ is given by
\begin{equation} \label{eq:cubic-ml-relax}
   \max\left\{ \sum_{i,j,k=1}^n a_{ijk}x_iy_jz_k: \|x\|_p \le 1, \|y\|_p \le 1, \|z\|_p \le 1 \right\}. 
\end{equation}
Thus, as far as approximating Problem (\ref{eq:generic-lp-opt}) is concerned, it suffices to focus on Problem (\ref{eq:cubic-ml-relax}).  Although the latter generally remains NP--hard (see Proposition \ref{prop:ML-NPh} and Theorem \ref{thm:MR-NPh}), intuitively it should be easier to handle because of the decoupling of variables.  Indeed, following the ideas in~\cite{KN08,S11a}, one can show that the optimal value of Problem (\ref{eq:cubic-ml-relax}) is equal to half times the $L_q$--diameter of a certain convex body $\mathcal{K}_p$, where $q=p/(p-1) \in [1,2]$ is the conjugate of $p$.  However, the latter quantity is known to be efficiently approximable only when $p=2$.  To tackle the case where $p>2$, we do not work on $\mathcal{K}_p$ directly as in~\cite{S11a}.  Instead, we construct another convex body $\mathcal{K}_p'$ whose $L_q$--diameter is within a constant factor of the optimal value of Problem (\ref{eq:cubic-ml-relax}) but can be approximated efficiently.  The validity of our construction is established using Grothendieck's inequality---a tool that originates from functional analysis and has since found many applications in optimization and theoretical computer science; see, e.g.,~\cite{SZY07,KN12,P12}.  Consequently, we are able to approximate Problem (\ref{eq:cubic-ml-relax}) and hence also Problem (\ref{eq:generic-lp-opt}) in polynomial time for the case where $d=3$.

To extend the above results to the case where $d>3$, a natural idea is to apply recursion.  We will present two implementations of this idea, which will lead to two algorithms with different characteristics.  The first is based on the following crucial observation (see Proposition~\ref{prop1}):  Suppose that we have a deterministic approximation algorithm $\mathscr{A}_d$ for optimizing a degree--$d$ multilinear form over $L_p$--balls, where $d\ge3$.  Consider a degree--$(d+1)$ multilinear form $F$.  For any $\bar{x}^1 \in \R^n$, let $G_d(\bar{x}^1)$ be the value returned by $\mathscr{A}_d$ when applied to the degree--$d$ multilinear optimization problem
$$ \max\left\{ F(\bar{x}^1,x^2,\ldots,x^{d+1}): \|x^i\|_p \le 1 \quad \mbox{for } i=2,3,\ldots,d+1 \right\}. $$
Then, the function $G_d$ essentially defines a norm on $\R^n$.  Such a property, which was first established in~\cite{S11a} for the case where $p=2$, is extremely useful and can be of independent interest.  In particular, it allows us to utilize existing polytopal approximations of $L_p$--balls~\cite{BGK+01} to design a deterministic $\Omega\left( (\log n)^{(d-2)/p} \big/ n^{d/2-1} \right)$--approximation algorithm for Problem (\ref{eq:generic-lp-opt}). 

The second approach to implementing the recursion idea is by randomization.  Specifically, consider a degree--$d$ multilinear form $F$, where $d>3$.  It is known that if $x^2,\ldots,x^d \in \R^n$ are arbitrary and $\xi \in \R^n$ is a random vector uniformly distributed on the $L_q$--sphere, then 
$$ F(\xi,x^2,\ldots,x^d) \ge \Omega\left( \sqrt{\frac{\log n}{n}} \right) \cdot \left[ \max_{\|x\|_p \le 1} F(x,x^2,\ldots,x^d) \right] $$
holds with a probability that is at least inversely proportional to a polynomial in $n$; cf.~\cite[Lemma 3.3]{KN08}.  Using this result, it is not hard to show that any $\beta_{d-1}$--approximation algorithm for optimizing a degree--$(d-1)$ multilinear form over $L_p$--balls will yield an $\Omega(\beta_{d-1}\sqrt{\log n/n})$--approximation algorithm for Problem (\ref{eq:generic-lp-opt}).  To complete the argument, we show by induction that $\beta_{d-1}$ can be taken as $\beta_{d-1}=\Omega((\log n/n)^{(d-1)/2-1})$.  This gives an $\Omega((\log n/n)^{d/2-1})$--approximation algorithm for Problem (\ref{eq:generic-lp-opt}).  It should be noted that unlike the deterministic algorithm described above, the algorithm obtained using this approach is randomized and thus will only attain the stated approximation ratio with high probability.  However, it is much easier to implement than its deterministic counterpart.

The rest of the paper is organized as follows.  Section \ref{sec:prelim} contains the preliminaries.   In Section~\ref{sec:lp-hardness}, we show that the problem of optimizing a homogeneous polynomial of fixed degree over an $L_p$--ball is NP--hard.  Then, in Section \ref{sec:MR}, we introduce a multilinear relaxation of the $L_p$--ball constrained homogeneous polynomial optimization problem and show that it is equivalent to the latter from an approximation perspective.  We also discuss the hardness of the multilinear relaxation.  In Section \ref{sec:ml-diam}, we develop both deterministic and randomized polynomial--time approximation algorithms for the problem of optimizing a multilinear form over $L_p$--balls by relating it to the problem of determining the diameters of certain convex bodies.  Finally, we conclude with some closing remarks in Section \ref{sec:concl}.

\section{Preliminaries} \label{sec:prelim}
We begin with the notation and definitions used in this paper.  A \emph{tensor} is a multidimensional array, and the \emph{order} of a tensor is the number of dimensions.  Let $\mathcal{A}=(a_{i_{1}i_{2}\cdots i_{d}})\in \R^{n_{1}\times n_{2}\times \cdots \times n_{d}}$ be a tensor of order \emph{d}.  We denote its $(i_1,i_2,\ldots,i_d)$--th element by either $a_{i_1i_2\cdots i_d}$ or $[\mathcal{A}]_{i_1i_2\cdots i_d}$.  We say that $\mathcal{A}$ is \emph{non--zero} if at least one of its elements is non--zero, and is \emph{cubical} if $n_{1}=n_{2}=\cdots=n_{d}$. A cubical tensor is said to be \emph{super--symmetric} if every element $a_{i_{1}i_{2}\cdots i_{d}}$ is invariant under any permutation of the indices.

Let $K$ and $j_{1},j_{2},\ldots,j_{K}$ be integers such that $1\leq K \leq d$ and $1\leq j_{1}< j_{2}<\cdots< j_{K}\leq d$. Furthermore, let $x^{j_{k}}\in \mathbb{R}^{n_{j_{k}}}$, where $k=1,\ldots,K$, be given vectors. We use $\mathcal{A}(x^{j_{1}},x^{j_{2}},\ldots,x^{j_{K}})$ to denote the order--$(d-K)$ tensor obtained by ``summing out'' the indices $j_{1},j_2,\ldots,j_{K}$ from the order--$d$ tensor $\mathcal{A}=(a_{i_{1}i_{2}\cdots i_{d}})\in \mathbb{R}^{n_{1}\times n_{2}\times \cdots \times n_{d}}$ using $x^{j_1},x^{j_2},\ldots,x^{j_K}$. For instance, if $K=2$, $j_{1}=2$ and $j_{2}=4$, then 
$$\mathcal{A}(x^{2},x^{4})_{i_{1}i_{3}i_{5}i_{6}\cdots i_{d}}=\sum_{i_{2}=1}^{n_{2}}\sum_{i_{4}=1}^{n_{4}}a_{i_{1}i_{2}\cdots i_{d}}x_{i_{2}}^{2}x_{i_{4}}^{4}. $$

Given an order--$d$ tensor $\mathcal{A}=(a_{i_{1}i_{2}\cdots i_{d}})\in \R^{n_{1}\times n_{2}\times \cdots \times n_{d}}$, we can associate with it a multilinear form $F_{\mathcal{A}}: \R^{n_{1}}\times \R^{n_{2}}\times \cdots \times \R^{n_{d}}\rightarrow \R$ via
$$F_{\mathcal{A}}(x^{1},x^{2},\ldots,x^{d})=\sum_{i_{1}=1}^{n_{1}}\cdots\sum_{i_{d}=1}^{n_{d}}a_{i_{1}i_{2}\cdots i_{d}}x_{i_{1}}^{1}x_{i_{2}}^{2} \cdots x_{i_{d}}^{d}. $$
If $\mathcal{A}$ is super--symmetric with $n_{1}=n_{2}=\cdots=n_{d}=n$, then we can further associate with it a homogeneous degree--$d$ polynomial $f_{\mathcal{A}}: \R^{n}\rightarrow \R$ via
$$ f_{\mathcal{A}}(x) = F_{\mathcal{A}}(x,x,\ldots,x)=\sum_{1\leq i_{1},\ldots,i_{d}\leq n}a_{i_{1}i_{2}\cdots i_{d}}x_{i_{1}}x_{i_{2}} \cdots x_{i_{d}}. $$
In general, even if $\mathcal{A}$ is not super--symmetric or even cubical, it is still possible to relate the multilinear form $F_\mathcal{A}$ to a certain homogeneous degree--$d$ polynomial via symmetrization~\cite{RvL11}.  To introduce this procedure, we need some preliminary definitions.  Let $\pi=(\pi_1,\pi_2,\ldots,\pi_d)$ be a permutation of the set $\{1,2,\ldots,d\}$.  The \emph{$\pi$--transpose} of $\mathcal{A}=(a_{i_1i_2\cdots i_d}) \in \R^{n_1\times n_2\times\cdots\times n_d}$ is the order--$d$ tensor $\mathcal{A}^{\pi} = (\bar{a}_{i_{\pi_1}i_{\pi_2}\cdots i_{\pi_d}}) \in \R^{n_{\pi_1} \times n_{\pi_2} \times \cdots \times n_{\pi_d}}$ whose elements are given by
$$ \bar{a}_{i_{\pi_1}i_{\pi_2}\cdots i_{\pi_d}} = a_{i_1i_2\cdots i_d} \quad\mbox{for } i_j = 1,\ldots,n_j; \, j=1,\ldots,d. $$
Let $N=n_1+n_2+\cdots+n_d$ and partition the index set $\{1,\ldots,N\}$ into sets of consecutive integers as follows:
\begin{equation} \label{eq:partition}
\{1,\ldots,N\} = \bigcup_{j=1}^d B_j, \quad\mbox{where } B_j = \left\{ \sum_{i=1}^{j-1} n_i + 1, \ldots, \sum_{i=1}^{j} n_i \right\}.
\end{equation}
Given an arbitrary cubical order--$d$ tensor $\mathcal{B} \in \R^{N^d}$ and $\chi_i \in \{1,\ldots,d\}$ for $i=1,\ldots,d$, the \emph{$(\chi_1,\ldots,\chi_d)$--th block} of $\mathcal{B}$ is defined as the sub--tensor
$$ \mathcal{B}_{\chi_1\chi_2\cdots\chi_d} = (b_{i_1i_2\cdots i_d})_{i_j \in B_{\chi_j};\, j=1,\ldots,d} \in \R^{n_{\chi_1}\times n_{\chi_2} \times \cdots \times n_{\chi_d}}. $$
Armed with these definitions, we define the \emph{symmetrization} of $\mathcal{A}=(a_{i_1i_2\cdots i_d}) \in \R^{n_1\times n_2\times \cdots \times n_d}$ as the order--$d$ cubical tensor $\mbox{sym}(\mathcal{A}) \in \R^{N^d}$ whose blocks are given by
$$ \left[ \mbox{sym}(\mathcal{A}) \right]_{\chi_1\chi_2\cdots\chi_d} = \left\{ 
\begin{array}{c@{\quad}l}
   \mathcal{A}^\chi & \mbox{if } \chi=(\chi_1,\chi_2,\ldots,\chi_d) \mbox{ is a permutation of } \{1,2,\ldots,d\}, \\
   \noalign{\medskip}
   \bz & \mbox{otherwise}.
\end{array}
\right.
$$
For instance, when $d=2$, $\mathcal{A}$ is an $n_1\times n_2$ matrix, and its symmetrization is given by the well--known construction
$$ \mbox{sym}(\mathcal{A}) = \left[ \begin{array}{cc} \bz & \mathcal{A} \\ \noalign{\smallskip} \mathcal{A}^T & \bz \end{array} \right]. $$
More generally, it is known that the tensor $\mbox{sym}(\mathcal{A})$ enjoys the following properties~\cite{RvL11}:
\begin{enumerate}
   \item $\mbox{sym}(\mathcal{A})$ is super--symmetric.  In particular, it can be associated with a homogeneous degree--$d$ polynomial $f_{\text{sym}(\mathcal{A})}$.

   \item For every $x = \left[ \, (x^1)^T \, (x^2)^T \, \cdots \, (x^d)^T \, \right]^T \in \R^N$, where $x^i \in \R^{n_i}$ and $i=1,\ldots,d$, we have 
\begin{equation} \label{eq:sym-id}
   f_{\text{sym}(\mathcal{A})}(x) = d! \cdot F_\mathcal{A}(x^1,x^2,\ldots,x^d).
\end{equation}
\end{enumerate}

Now, let $d \geq 3$ and $p \in [2,\infty]$ be given. Let $\mathcal{A}=(a_{i_{1}i_{2}\cdots i_{d}})\in \R^{n^{d}}$ be an arbitrary non--zero super--symmetric tensor of order $d$, and let $f_{\mathcal{A}}: \R^{n}\rightarrow \R$ be the corresponding homogeneous polynomial.  Our main objective in this paper is to study the algorithmic aspects of the following $L_p$--ball constrained homogeneous polynomial optimization problem:
$$
({\sf HP}) \qquad
\begin{array}{ccc@{\quad}l}
   \bar{v} &=& \mbox{maximize} & \displaystyle{ f_{\mathcal{A}}(x)\equiv \sum_{1\leq i_{1},\ldots,i_{d}\leq n}a_{i_{1}i_{2}\cdots i_{d}}x_{i_{1}}x_{i_{2}}\cdots x_{i_{d}} } \\
   \noalign{\medskip}
   & & \mbox{subject to} & \|x\|_p \le 1, \, x \in \R^n.
\end{array}
$$

\section{Hardness of $L_p$--Ball Constrained Homogeneous Polynomial Optimization} \label{sec:lp-hardness}
We begin with the following result, which concerns the complexity of Problem $({\sf HP})$:
\begin{thm} \label{thm:HP-NPh}
Problem $({\sf HP})$ is NP--hard for any given $d\ge3$ and $p\in[2,\infty]$.
\end{thm}
The proof of Theorem \ref{thm:HP-NPh} consists of two steps.  First, we show that the problem of maximizing a degree--$d$ {\it multilinear form} over $L_p$--balls is NP--hard for any given $d\ge3$ and $p\in[2,\infty]$.  Then, we give a polynomial--time reduction of this problem to Problem $({\sf HP})$ using the symmetrization procedure introduced in Section \ref{sec:prelim}, thereby proving the NP--hardness of the latter.

To begin, let us formally define the problem used in the first step.

\begin{center}
\begin{tabular}{|c@{\qquad}l|} \hline
   & Let $\mathcal{A}=(a_{i_1i_2\cdots i_d}) \in \R^{n_1\times n_2\times \cdots \times n_d}$ be an arbitrary non--zero order--$d$ tensor, and \\
   \multirow{3}{*}{$({\sf ML})$} & let $F_\mathcal{A}:\R^{n_1}\times\R^{n_2}\times\cdots\times\R^{n_d} \limto \R$ be the corresponding multilinear form. Solve \\
   & \vspace{-0.3cm} \\
   & \multicolumn{1}{c|}{$\begin{array}{ccc@{\quad}l} v_{\sf ML}(\mathcal{A},d) &=& \mbox{maximize} & F_\mathcal{A}(x^1,x^2,\ldots,x^d) \\
   \noalign{\medskip}
   & & \mbox{subject to} & \|x^{i}\|_{p} \le 1, \, x^i \in \R^{n_i} \quad \mbox{for } i = 1,\ldots,d.
\end{array}$} \\ \hline
\end{tabular}
\end{center}

\noindent We then have the following result:

\begin{prop} \label{prop:ML-NPh}
   Problem $({\sf ML})$ is NP--hard for any given $d\ge3$ and $p\in[2,\infty]$.
\end{prop}
\begin{proof}
Let $d\ge3$ and $p\in[2,\infty]$ be fixed.  Consider first the case where $p\in(2,\infty]$.  Our plan is to reduce the following problem---which is known to be NP--hard~\cite{Steinberg05}---to Problem $({\sf ML})$:

\begin{center}
\begin{tabular}{|c@{\qquad}l|} \hline
   & Let $B \in \R^{m \times n}$ and $p \in (2,\infty]$ be given. Let $q=p/(p-1)$ be the \\
   $({\sf NORM})$ & conjugate of $p$.  Compute $\|B\|_{p\limto q}$, the $p\limto q$ norm of $B$, where \\
   & \multicolumn{1}{c|}{$\|B\|_{p\limto q} = \max\{ \|By\|_q : \|y\|_p \le 1 \}$.} \\ \hline
\end{tabular}
\end{center}
\smallskip
\noindent Towards that end, suppose that we are given an instance of Problem $({\sf NORM})$.  By H\"{o}lder's inequality, we have
\begin{eqnarray*}
   \max_{\|y\|_p \le 1} \|By\|_q &=& \max_{\|x\|_p \le 1, \, \|y\|_p \le 1} x^TBy \\
   \noalign{\medskip}
   &=& \max_{ {\|x\|_p \le 1, \, \|y\|_p \le 1} \atop {|z_1|,\ldots,|z_{d-2}| \le 1}} \left( \prod_{i=1}^{d-2} z_i \right) x^TBy \\
   \noalign{\medskip}
   &=& \max_{ {\|x\|_p \le 1, \, \|y\|_p \le 1} \atop { |z_1|,\ldots,|z_{d-2}| \le 1}} F_\mathcal{A}\left( z_1,z_2,\ldots,z_{d-2},x,y \right),
\end{eqnarray*}
where $\mathcal{A} = (a_{1,\ldots,1,i,j}) \in \R^{1\times\cdots\times1\times m \times n}$ is the order--$d$ tensor with $a_{1,\ldots,1,i,j} = b_{ij}$ for $i=1,\ldots,m$ and $j=1,\ldots,n$; $F_\mathcal{A}:\R\times\cdots\times\R\times\R^m\times\R^n \limto \R$ is the multilinear form associated with $\mathcal{A}$.  This establishes the NP--hardness of Problem $({\sf ML})$ when $d\ge3$ and $p\in(2,\infty]$.

Next, consider the case where $p=2$.  It has been shown in~\cite[Proposition 2]{HLZ10} that Problem $({\sf ML})$ is NP--hard when $d=3$ and $p=2$.  Now, let $\mathcal{B}=(b_{ijk}) \in \R^{n_1\times n_2\times n_3}$ be an arbitrary non--zero order--$3$ tensor, and let $F_\mathcal{B}:\R^{n_1}\times\R^{n_2}\times\R^{n_3}\limto\R$ be the corresponding multilinear form.  Using similar argument as above, for any given $d\ge4$, we have
\begin{eqnarray*}
\max_{\|w\|_2,\,\|x\|_2,\,\|y\|_2 \le 1} F_\mathcal{B}(w,x,y) &=& \max_{ {\|w\|_2,\,\|x\|_2,\,\|y\|_2 \le 1} \atop {|z_1|,\ldots,|z_{d-3}| \le 1} } \left( \prod_{i=1}^{d-3} z_i \right) F_\mathcal{B}(w,x,y) \\
\noalign{\medskip}
&=& \max_{ {\|w\|_2,\,\|x\|_2,\,\|y\|_2 \le 1} \atop {|z_1|,\ldots,|z_{d-3}| \le 1} } F_\mathcal{A}(z_1,z_2,\ldots,z_{d-3},w,x,y),
\end{eqnarray*}
where $\mathcal{A} = (a_{1,\ldots,1,i,j,k}) \in \R^{1\times\cdots\times1\times n_1 \times n_2 \times n_3}$ is the order--$d$ tensor with $a_{1,\ldots,1,i,j,k} = b_{ijk}$ for $i=1,\ldots,n_1$, $j=1,\ldots,n_2$ and $k=1,\ldots,n_3$; $F_\mathcal{A}:\R\times\cdots\times\R\times\R^{n_1}\times\R^{n_2}\times\R^{n_3} \limto \R$ is the multilinear form associated with $\mathcal{A}$.  Thus, we conclude that when $p=2$, Problem $({\sf ML})$ remains NP--hard for each fixed $d\ge3$.
\end{proof}

\medskip
Next, we have the following proposition, which links the optimization of the multilinear form associated with a tensor $\mathcal{A}$ to that of the homogeneous polynomial associated with $\mbox{sym}(\mathcal{A})$.
\begin{prop} \label{prop:ml-sym-eqv}
   Let $d\ge2$ and $p \in [2,\infty]$ be given, and let $\mathcal{A} \in \R^{n_1\times n_2\times \cdots \times n_d}$ be an arbitrary non--zero order--$d$ tensor.  Set $N=n_1+n_2+\cdots+n_d$.  Consider the optimization problems
\begin{equation} \label{eq:ml-opt}
\begin{array}{c@{\quad}l}
\mbox{maximize} & d! \cdot F_\mathcal{A}(x^1,x^2,\ldots,x^d) \\
\noalign{\medskip}
\mbox{subject to} & \|x^i\|_p \le 1, \, x^i \in \R^{n_i} \quad\mbox{for } i=1,\ldots,d
\end{array}
\end{equation}
and
\begin{equation} \label{eq:sym-opt}
\begin{array}{c@{\quad}l}
\mbox{maximize} & f_{\text{sym}(\mathcal{A})}(z) \\
\noalign{\medskip}
\mbox{subject to} & \|z\|_p \le d^{1/p}, \, z \in \R^N,
\end{array}
\end{equation}
where $F_\mathcal{A}$ is the multilinear form associated with $\mathcal{A}$ and $f_{\text{sym}(\mathcal{A})}$ is the homogeneous polynomial associated with the symmetrization of $\mathcal{A}$ (see Section \ref{sec:prelim}).  Let $(\bar{x}^1,\bar{x}^2,\ldots,\bar{x}^d) \in \R^{n_1} \times \R^{n_2} \times \cdots \times \R^{n_d}$ and $\bar{z}=\left[ \, (\bar{z}^1)^T \, (\bar{z}^2)^T \, \cdots \, (\bar{z}^d)^T \, \right]^T$, where $\bar{z}^i \in \R^{n_i}$ for $i=1,\ldots,d$, be optimal solutions to problems (\ref{eq:ml-opt}) and (\ref{eq:sym-opt}), respectively.  Then, the following hold:
\begin{enumerate}
\item[\subpb] $\|\bar{z}^i\|_p=1$ for $i=1,\ldots,d$.

\item[\subpb] $(\bar{z}^1,\bar{z}^2,\ldots,\bar{z}^d) \in \R^{n_1\times n_2\times \cdots \times n_d}$ and $\bar{x} = \left[ \, (\bar{x}^1)^T \, (\bar{x}^2)^T \, \cdots \, (\bar{x}^d)^T \, \right]^T \in \R^N$ are optimal solutions to problems (\ref{eq:ml-opt}) and (\ref{eq:sym-opt}), respectively.
\end{enumerate}
\resetspb
\end{prop}
\begin{proof}
Let us first consider the case where $p=\infty$.  By (\ref{eq:sym-id}), Problem (\ref{eq:sym-opt}) is equivalent to
$$
\begin{array}{c@{\quad}l}
\mbox{maximize} & d! \cdot F_{\mathcal{A}}(z^1,z^2,\ldots,z^d) \\
\noalign{\medskip}
\mbox{subject to} & \|z^i\|_\infty \le 1, \, z^i \in \R^{n_i} \quad\mbox{for } i=1,\ldots,d,
\end{array}
$$
which has exactly the same form as Problem (\ref{eq:ml-opt}).  Thus, the desired results follow immediately.

Now, consider the case where $p\in[2,\infty)$.  To prove (a), we again appeal to (\ref{eq:sym-id}), which implies the equivalence of Problem (\ref{eq:sym-opt}) and the following problem:
$$
\begin{array}{c@{\quad}l}
\mbox{maximize} & d! \cdot F_{\mathcal{A}}(z^1,z^2,\ldots,z^d) \\
\noalign{\medskip}
\mbox{subject to} & \displaystyle{ \sum_{i=1}^d \|z^i\|_p^p \le d, } \\
\noalign{\medskip}
& z^i \in \R^{n_i} \quad\mbox{for } i=1,\ldots,d.
\end{array}
$$
Since $\mathcal{A}$ is non--zero, we must have $\sum_{i=1}^d \|\bar{z}^i\|_p^p = d$ and $\|\bar{z}^i\|_p > 0$ for $i=1,\ldots,d$.  Now, suppose that $\|\bar{z}^j\|_p^p = \theta \not= 1$ for some $j \in \{1,\ldots,d\}$.  Then, we have $\sum_{i\not=j} \|\bar{z}^i\|_p^p = d-\theta > 0$.  Upon setting
$$ \hat{z}^i = \left\{
\begin{array}{r@{\quad}l}
\displaystyle{ \left( \frac{d-1}{d-\theta} \right)^{1/p} \bar{z}^i } & \mbox{if } i \not= j, \\
\noalign{\medskip}
\displaystyle{ \theta^{-1/p} \bar{z}^j } & \mbox{otherwise},
\end{array}
\right.
$$
we obtain
$$ \|\hat{z}^j\|_p^p = 1, \quad \sum_{i=1}^d \|\hat{z}^i\|_p^p = \frac{d-1}{d-\theta} \sum_{i\not=j} \|\bar{z}^i\|_p^p + \|\hat{z}^j\|_p^p = d $$
and
\begin{equation} \label{eq:scaled}
F_\mathcal{A}(\hat{z}^1,\hat{z}^2,\ldots,\hat{z}^d) = (d-1)^{\frac{d-1}{p}} \cdot \left( (d-\theta)^{d-1}\theta \right)^{-1/p} \cdot F_\mathcal{A}(\bar{z}^1,\bar{z}^2,\ldots,\bar{z}^d).
\end{equation}
In particular, we see that $\hat{z}=\left[ \, (\hat{z}^1)^T \, (\hat{z}^2)^T \, \cdots \, (\hat{z}^d)^T \, \right]^T \in \R^N$ is feasible for (\ref{eq:sym-opt}).  It is easy to verify that the function $t \mapsto ((d-t)^{d-1}t))^{-1/p}$ is strictly convex on $(0,d)$ and is minimized at $t=1$.  Since $\theta\not=1$ and $\mathcal{A}$ is non--zero, it follows from (\ref{eq:scaled}) that $F_\mathcal{A}(\hat{z}^1,\hat{z}^2,\ldots,\hat{z}^d) > F_\mathcal{A}(\bar{z}^1,\bar{z}^2,\ldots,\bar{z}^d)$, which contradicts the optimality of $\bar{z}$.  Thus, we have $\|\bar{z}^i\|_p^p=1$ for $i=1,\ldots,d$, as desired.

To prove (b), we first observe that since $\bar{x} = \left[ \, (\bar{x}^1)^T \, (\bar{x}^2)^T \, \cdots \, (\bar{x}^d)^T \, \right]^T \in \R^N$ is feasible for Problem (\ref{eq:sym-opt}) and $f_{\text{sym}(\mathcal{A})}(\bar{x}) = d!\cdot F_\mathcal{A}(\bar{x}^1,\bar{x}^2,\ldots,\bar{x}^d)$ by (\ref{eq:sym-id}), we have $f_{\text{sym}(\mathcal{A})}(\bar{z}) \ge d!\cdot F_\mathcal{A}(\bar{x}^1,\bar{x}^2,\ldots,\bar{x}^d)$.  Now, the result in (a) implies that $(\bar{z}^1,\bar{z}^2,\ldots,\bar{z}^d) \in \R^{n_1\times n_2\times \cdots \times n_d}$ is feasible for Problem (\ref{eq:ml-opt}), and hence using (\ref{eq:sym-id}) we obtain $f_{\text{sym}(\mathcal{A})}(\bar{z}) = d!\cdot F_\mathcal{A}(\bar{z}^1,\bar{z}^2,\ldots,\bar{z}^d) \le d!\cdot F_\mathcal{A}(\bar{x}^1,\bar{x}^2,\ldots,\bar{x}^d)$.  This completes the proof.
\end{proof}

\medskip
We are now ready to complete the proof of Theorem \ref{thm:HP-NPh}.

\medskip
\noindent{\bf Proof of Theorem \ref{thm:HP-NPh}}\,\,\, Proposition \ref{prop:ml-sym-eqv} implies that Problem $({\sf ML})$ is equivalent to 
$$
\begin{array}{c@{\quad}l}
\mbox{maximize} & f_{\text{sym}(\mathcal{A})}(z) \\
\noalign{\medskip}
\mbox{subject to} & \|z\|_p \le 1, \, z \in \R^N,
\end{array}
$$
where $N=n_1+n_2+\cdots+n_d$.  The latter is clearly an instance of Problem $({\sf HP})$.  Moreover, when $d\ge3$ is fixed, the size of $\mbox{sym}(\mathcal{A})$ is polynomial in $n_1,n_2,\ldots,n_d$.  Thus, for any given $d\ge3$ and $p\in[2,\infty]$, we can reduce Problem $({\sf ML})$ to Problem $({\sf HP})$ in polynomial time, which implies that the latter is NP--hard, as desired. \endproof


\section{$L_{p}$--Ball Constrained Homogeneous Polynomial Optimization and Its Multilinear Relaxation} \label{sec:MR}
In view of Theorem \ref{thm:HP-NPh}, we now turn our attention to the task of designing polynomial--time approximation algorithms for Problem $({\sf HP})$ with provable guarantees.  Towards that end, consider the following multilinear relaxation of Problem $({\sf HP})$:
$$
({\sf MR}) \qquad
\begin{array}{ccc@{\quad}l}
   v^* &=& \mbox{maximize} & \displaystyle{ F_{\mathcal{A}}(x^{1},x^{2},\ldots,x^{d}) \equiv \sum_{1\leq i_{1},\ldots,i_{d}\leq n}a_{i_{1}i_{2}\cdots i_{d}}x_{i_{1}}^{1}x_{i_{2}}^{2}\cdots x_{i_{d}}^{d} } \\
   \noalign{\medskip}
   & & \mbox{subject to} & \|x^{i}\|_{p} \le 1, \, x^i \in \R^{n_i} \quad \mbox{for } i = 1,\ldots,d.
\end{array}
$$
Since $f_\mathcal{A}(x)=F_\mathcal{A}(x,x,\ldots,x)$ for all $x\in\R^n$ and $x=\bz$ is feasible for Problem $({\sf HP})$, we clearly have $v^* \ge \bar{v} \ge 0$.  Our motivation for studying Problem $({\sf MR})$ comes from the following result, which essentially states that $v^*$ and $\bar{v}$ are within a constant factor of each other when $d\ge3$ is fixed.
\begin{thm}\label{hpml}
Let $d\ge 3$ and $p \in [2,\infty]$ be given.  Suppose there is a polynomial--time algorithm $\mathscr{A}_{\sf MR}$ that, given any instance of Problem $({\sf MR})$, returns a feasible solution whose objective value is at least $\alpha v^*$ for some $\alpha \in(0,1]$.  Then, there is a polynomial--time algorithm $\mathscr{A}_{\sf HP}$ that, given any instance of Problem $({\sf HP})$, returns a solution $\hat{x}\in \R^{n}$ with $\| \hat{x} \|_{p}\leq 1$ and
$$
\begin{array}{rcl@{\qquad}l}
  f_\mathcal{A}(\hat{x}) &\ge& \alpha\cdot d!\cdot d^{-d}\cdot v^* \\
  \noalign{\medskip}
  &\ge& \alpha\cdot d!\cdot d^{-d}\cdot \bar{v} & \mbox{for odd } d\ge3, \\
  \noalign{\bigskip}
  f_\mathcal{A}(\hat{x}) - \underline{v} &\ge& 2\alpha\cdot d!\cdot d^{-d}\cdot v^* \\
  \noalign{\medskip}
  &\ge& \alpha\cdot d!\cdot d^{-d}\cdot (\bar{v} - \underline{v}) & \mbox{for even } d\ge4,
\end{array}
$$
where $\underline{v} = \min_{\|x\|_p \le 1} f_\mathcal{A}(x)$.  In other words, the algorithm $\mathscr{A}_{\sf HP}$ has an \emph{approximation guarantee} (resp.~\emph{relative approximation guarantee}) of $\alpha \cdot d! \cdot d^{-d}$ when $d$ is odd (resp.~even).
\end{thm}
For a proof of Theorem \ref{hpml}, see Appendix \ref{app:hpml}.  We remark that for the case where $p=2$, an analogous result has been established in~\cite{HLZ10}; cf.~\cite[Theorem 1]{S11a}.

Theorem \ref{hpml} shows that any algorithm for solving Problem $({\sf MR})$ will translate into an algorithm for approximating Problem $({\sf HP})$.  Although it seems intuitive that Problem $({\sf MR})$ is NP--hard as well, such a result does not follow directly from Proposition~\ref{prop:ML-NPh}, as the tensor associated with the objective function in Problem $({\sf ML})$ is not required to be super--symmetric or even cubical.  The following theorem fills this gap:
\begin{thm} \label{thm:MR-NPh}
   Problem $({\sf MR})$ is NP--hard for any given $d\ge3$ and $p\in[2,\infty]$.
\end{thm}
The proof of Theorem \ref{thm:MR-NPh} is quite involved and can be found in Appendix \ref{app:MR-NPh}.  We remark that Theorem \ref{thm:MR-NPh} is, to the best of our knowledge, the first hardness result for the problem of optimizing a {\it super--symmetric} multilinear form that holds for {\it any} given $d\ge3$ and $p\in[2,\infty]$; cf.~\cite{HL09}.

\section{$L_p$--Ball Constrained Multilinear Optimization and Diameters of Convex Bodies} \label{sec:ml-diam}
Given that both Problem $({\sf MR})$ and Problem $({\sf ML})$ are NP--hard, we shall study the slightly more general Problem $({\sf ML})$, where the focus will be on developing approximation algorithms with provable guarantees.  Since the case where $p=2$ has already been investigated in~\cite{S11a}, we shall assume that $p\in(2,\infty]$ in the sequel.




\subsection{Base Case: Approximating $L_p$--Ball Constrained Trilinear Maximization} \label{sec:base}
Let us begin by considering the case where $d=3$.  Specifically, let $\mathcal{A}=(a_{ijk})\in \mathbb{R}^{n_{1}\times n_{2}\times n_{3}}$ be an arbitrary non--zero order--3 tensor.  Without loss of generality, we assume that $1\leq n_{1}\leq n_{2}\leq n_{3}$.  Then, Problem $({\sf ML})$ becomes
\begin{equation} \label{eq:trilinear}
\begin{array}{ccc@{\quad}l}
   v_{\sf ML}(\mathcal{A},3) &=& \mbox{maximize} & \displaystyle{ \sum_{i=1}^{n_1} \sum_{j=1}^{n_2} \sum_{k=1}^{n_3} a_{ijk} x_i^1 x_j^2 x_k^3 } \\
   \noalign{\medskip}
   & & \mbox{subject to} & \|x^{i}\|_{p} \le 1, \, x^i \in \R^{n_i} \quad \mbox{for } i = 1,2,3.
\end{array}
\end{equation}
Using the definition of $\mathcal{A}(x^1)$ and H\"{o}lder's inequality, we can express $v_{\sf ML}(\mathcal{A},3)$ as
\begin{eqnarray}
   v_{\sf ML}(\mathcal{A},3) &=& \max_{\|x^1\|_p \le 1} \max_{\|x^2\|_p \le 1,\,\|x^3\|_p \le 1} (x^2)^T \mathcal{A}(x^1) x^3 \nonumber \\
   \noalign{\medskip}
   &=& \max_{\|x^1\|_p \le 1} \max_{\|x^3\|_p \le 1} \| \mathcal{A}(x^1)x^3 \|_q \nonumber \\
   \noalign{\medskip}
   &=& \max_{\|x^1\|_p \le 1} \| \mathcal{A}(x^1) \|_{p\limto q}, \label{maxmax}
\end{eqnarray}
where $q=p/(p-1)$ is the conjugate of $p$ and $\|\mathcal{A}(x^1)\|_{p \limto q}$ is the $p \limto q$ norm of the $n_2\times n_3$ matrix $\mathcal{A}(x^1)$.  From the above derivation, we see that Problem (\ref{eq:trilinear}) encapsulates two difficult computational tasks: (i) the computation of $\|\mathcal{A}(x^1)\|_{p\limto q}$ for any given $x^1 \in \R^{n_1}$, and (ii) the maximization of a convex function $x^1 \mapsto \|\mathcal{A}(x^1)\|_{p\limto q}$ over a convex set $B_p^{n_1}=\{ x \in \R^{n_1}: \|x\|_p \le 1\}$.  To tackle these difficulties, we proceed in two steps.  First, we show that $\|\mathcal{A}(x^1)\|_{p\limto q}$ can be approximated by another efficiently computable norm.  Then, we show that the maximization of this latter norm over $B_p^{n_1}$ is equivalent to determining the $L_q$--diameter of a certain convex body, a problem for which approximation algorithms are available.  This would in turn yield approximation algorithms for Problem (\ref{eq:trilinear}).

\medskip
\noindent {\bf Step 1: Approximating $\|B\|_{p\limto q}$ when $p \in (2,\infty]$.} The task of computing $\|B\|_{p\limto q}$ for any given $m\times n$ matrix $B$ and $p\in(2,\infty]$ is an instance of the matrix norm problem, which has been extensively studied in the literature.  In particular, Nesterov~\cite{N00} showed that $\|B\|_{p\limto q}$ can be approximated to within a factor of $\frac{2\sqrt{3}}{\pi}-\frac{2}{3}>0.435$ via a certain convex relaxation.  Later, Ben--Tal and Nemirovski~\cite{BN01b} and Steinberg~\cite{Steinberg05} established the NP--hardness of the problem and gave a more refined analysis of Nesterov's relaxation scheme.  However, the approximation bound they obtained is better than Nesterov's only when the parameters $m,n,p$ belong to a certain regime.  As it turns out, by considering a different convex relaxation, it is possible to obtain an approximation bound that uniformly improves upon that of Nesterov.  To demonstrate this, we first observe that
\begin{equation} \label{eq:pqnorm-max}
\begin{array}{ccc@{\quad}l}
   \|B\|_{p \limto q} &=& \mbox{maximize} & 
   \displaystyle{
   \frac{1}{2} \left[\begin{array}{cc} \bz & B \\ B^T & \bz \end{array}\right] \bullet \left[\begin{array}{c} y \\ z \end{array}\right] \left[\begin{array}{cc} y^T & z^T \end{array}\right] } \\
   \noalign{\medskip}
   && \mbox{subject to} & \|y\|_p \le 1, \, \|z\|_p \le 1,
\end{array}
\end{equation}
where $P \bullet Q = \mbox{tr}(P^TQ)$ denotes the Frobenius inner product of the matrices $P$ and $Q$.  Hence, by introducing the $(m+n)\times(m+n)$ positive semidefinite (psd) matrix $X$ to replace the rank--one psd matrix $(y,z)(y,z)^T$ and denoting
$$ \tilde{B} = \frac{1}{2} \left[\begin{array}{cc} \bz & B \\ B^T & \bz \end{array}\right], $$
we obtain the following relaxation of $\|B\|_{p \limto q}$:
\begin{equation} \label{eq:pqnorm-cvx}
{\sf vec}_p(B) = \left\{
\begin{array}{l@{\quad}l}
\displaystyle{ \max\left\{ \tilde{B} \bullet X : \sum_{i=1}^m |X_{ii}|^{p/2} \le 1, \, \sum_{j=m+1}^{m+n} |X_{jj}|^{p/2} \le 1, \, X \succeq \bz \right\} } & \mbox{for } p \in (2,\infty), \\
\noalign{\medskip}
\displaystyle{ \max\left\{ \tilde{B} \bullet X : \max_{1\le i\le m} |X_{ii}| \le 1, \, \max_{m+1\le j \le m+n} |X_{jj}| \le 1, \, X \succeq \bz \right\} } & \mbox{for } p = \infty.
\end{array}
\right.
\end{equation}

Note that for $p>2$, Problem (\ref{eq:pqnorm-cvx}) is a convex program that can be solved to arbitrary accuracy in polynomial time using, e.g., the ellipsoid method~\cite{GLS93} (cf.~\cite{KNS10}).  Moreover, the following simple observation of Khot and Naor~\cite{KN12} shows that the ratio between ${\sf vec}_p(B)$ and $\|B\|_{p \limto q}$ is bounded above by the Grothendieck constant $K_G$, which is known to be strictly less than $\frac{\pi}{2\ln(1+\sqrt{2})} < 1.783$~\cite{BMMN11}.  
\begin{prop} \label{prop:grothendieck}
   The following inequalities hold:
$$ \|B\|_{p \limto q} \le {\sf vec}_p(B) \le K_G \cdot \|B\|_{p \limto q}. $$
\end{prop}
For completeness, we include the proof of Proposition \ref{prop:grothendieck} here.

\medskip
\noindent\begin{proof}
   The first inequality follows readily from the fact that Problem (\ref{eq:pqnorm-cvx}) is a relaxation of Problem (\ref{eq:pqnorm-max}).  To prove the second inequality, consider an optimal solution $X^*$ to Problem (\ref{eq:pqnorm-cvx}) with $\mbox{rank}(X^*)=r\ge1$.  Let $X^*=V^TV$, where $V \in \R^{r\times(m+n)}$, be the Cholesky factorization of $X^*$.  Furthermore, let $u_i \in \R^r$ (where $i=1,\ldots,m$) and $v_j \in \R^r$ (where $j=1,\ldots,n$) be the $i$--th column and $(m+j)$--th column of $V$, respectively.  Then, by the optimality of $X^*$, we have
$$ {\sf vec}_p(B) = \tilde{B} \bullet X^* = \sum_{i=1}^m\sum_{j=1}^n B_{ij} u_i^Tv_j. $$
Moreover, since $\mbox{diag}(\tilde{B})=\bz$, we may assume that 
$$
\begin{array}{r@{\quad}l}
\|u_i\|_2 = |X_{ii}^*|^{1/2} = 1 & \mbox{for } 1\le i\le m, \\
\noalign{\medskip}
\|v_j\|_2 = |X_{jj}^*|^{1/2} = 1 & \mbox{for } m+1\le j\le m+n 
\end{array}
$$
in the case where $p=\infty$, or
$$ \sum_{i=1}^m \|u_i\|_2^p = \sum_{i=1}^m |X_{ii}^*|^{p/2} = 1, \quad \sum_{j=1}^n \|v_j\|_2^p = \sum_{j=m+1}^{m+n} |X_{jj}^*|^{p/2} = 1 $$
in the case where $p \in (2,\infty)$.  Now, define an $m \times n$ matrix $Q$ by $Q_{ij} = B_{ij}\cdot\|u_i\|_2\cdot\|v_j\|_2$, where $i=1,\ldots,m$ and $j=1,\ldots,n$.  By the Grothendieck inequality (see, e.g.,~\cite{AN06,KN12}), there exist vectors $\eta \in \{-1,1\}^m$, $\gamma \in \{-1,1\}^n$ such that
\begin{equation} \label{eq:grothendieck}
   {\sf vec}_p(B) = \sum_{i=1}^m \sum_{j=1}^n B_{ij} u_i^Tv_j = \sum_{i=1}^m \sum_{j=1}^n Q_{ij} \frac{u_i^Tv_j}{\|u_i\|_2\cdot\|v_j\|_2} \le K_G \sum_{i=1}^m \sum_{j=1}^n Q_{ij} \eta_i \gamma_j.
\end{equation}
Upon letting $\bar{y}_i=\eta_i\cdot\|u_i\|_2$ for $i=1,\ldots,m$ and $\bar{z}_j = \gamma_j\cdot\|v_j\|_2$ for $j=1,\ldots,n$, we see that $\|\bar{y}\|_p = \|\bar{z}\|_p = 1$ for $p \in (2,\infty]$;
i.e., $(\bar{y},\bar{z}) \in \R^m \times \R^n$ is feasible for Problem (\ref{eq:pqnorm-max}).  Moreover, we obtain from (\ref{eq:grothendieck}) that
$$ {\sf vec}_p(B) \le K_G \cdot \sum_{i=1}^m \sum_{j=1}^n B_{ij}\bar{y}_i\bar{z}_j \le K_G \cdot \|B\|_{p \limto q}. $$
This completes the proof.
\end{proof}

\medskip
\noindent The proof of Proposition \ref{prop:grothendieck} reveals that known algorithmic implementations of the Grothendieck inequality (see, e.g.,~\cite{AN06,BMMN11}) can be used to deliver vectors $\bar{y} \in \R^m$ and $\bar{z} \in \R^n$ that are feasible for Problem (\ref{eq:pqnorm-max}) and whose associated objective value $\bar{y}^TB\bar{z}$ is within a constant factor of $\|B\|_{p \limto q}$.  It should be noted, however, that the precise constant will depend on the particular implementation used.  For our purposes, we shall consider two different implementations of the Grothendieck inequality.  The first is a deterministic procedure introduced in~\cite{AN06}, which is based on the construction of small sample spaces with many four--wise independent random variables.  It guarantees that $K_G \le 27$, and hence by Proposition \ref{prop:grothendieck} there is a deterministic $(1/27)$--approximation algorithm for computing $\|B\|_{p \limto q}$.  Although the above procedure does not yield the best approximation bound for $\|B\|_{p \limto q}$ (in fact, it is even worse than Nesterov's bound), it will allow us to design a deterministic approximation algorithm for Problem $({\sf ML})$.  The second one is based on the so--called Krivine rounding scheme in~\cite{BMMN11}.  The resulting procedure is randomized and guarantees that $K_G < \frac{\pi}{2\ln(1+\sqrt{2})}$, which is currently the best bound on $K_G$.  Consequently, we can approximate $\|B\|_{p \limto q}$ to within a factor that is strictly larger than $\frac{2\ln(1+\sqrt{2})}{\pi}>0.561$, which is better than Nesterov's bound of $0.435$.

Based on the above discussion, we summarize our procedure for approximating $\|B\|_{p \limto q}$ in Algorithm \ref{alg:pq-norm}.

\begin{algorithm}
\caption{Procedure for Approximating $\|B\|_{p \limto q}$ when $p\in(2,\infty]$ and $q=p/(p-1)$}
\label{alg:pq-norm}
\begin{algorithmic}[1]
   \renewcommand{\algorithmicrequire}{\textbf{Input:}}
   \renewcommand{\algorithmicensure}{\textbf{Output:}}
   \renewcommand{\algorithmicforall}{\textbf{for each}}
   \REQUIRE An $m\times n$ matrix $B$, a rational number $p\in(2,\infty]$.
   \ENSURE A feasible solution $(\bar{y},\bar{z}) \in \R^m\times\R^n$ to Problem (\ref{eq:pqnorm-max}).
   \STATE Solve the convex relaxation (\ref{eq:pqnorm-cvx}) and let $X^*=V^TV$ be an optimal solution.  Let $u_1,\ldots,u_m$ and $v_1,\ldots,v_n$ be the first $m$ and last $n$ columns of $V$, respectively.

   \STATE Apply either the {\bf deterministic} rounding procedure in~\cite{AN06} or the {\bf randomized} rounding procedure in~\cite{BMMN11} to the vectors $\{u_i/\|u_i\|_2\}_{i=1}^m$ and $\{v_j/\|v_j\|_2\}_{j=1}^n$ to obtain vectors $\eta \in \{-1,1\}^m$ and $\gamma \in \{-1,1\}^n$ that satisfy (\ref{eq:grothendieck}), where $K_G \le 27$ if the deterministic procedure in~\cite{AN06} is used and $K_G < \frac{\pi}{2\ln(1+\sqrt{2})}$ if the randomized procedure in~\cite{BMMN11} is used.

   \STATE Set $\bar{y}_i = \eta_i \cdot \|u\|_2$ for $i=1,\ldots,m$ and $\bar{z}_j = \gamma_j \cdot \|v_j\|_2$ for $j=1,\ldots,n$.  Return $(\bar{y},\bar{z}) \in \R^m \times \R^n$.
\end{algorithmic}
\end{algorithm}

\medskip
\noindent{\bf Step 2: Norm Maximization and Diameters of Convex Bodies.} In view of (\ref{maxmax}) and Proposition \ref{prop:grothendieck}, we see that any $\alpha$--approximation to
\begin{equation} \label{eq:max-norm-cubic}
   \max_{\|x^1\|_p \le 1} {\sf vec}_p\left( \mathcal{A}(x^1) \right)
\end{equation}
will yield an $(\alpha/K_G)$--approximation to $v_{\sf ML}(\mathcal{A},3)$.  Hence, it suffices to focus on Problem (\ref{eq:max-norm-cubic}).  The following result shows that Problem (\ref{eq:max-norm-cubic}) is in fact equivalent to maximizing a certain norm over the $L_p$--ball.
\begin{prop} \label{prop:norm}
   Let $\mathcal{A} = (a_{ijk}) \in \R^{n_1\times n_2 \times n_3}$ be an arbitrary non--zero order--$3$ tensor.  Consider the $(n_2\times n_3) \times n_1$ matrix $A$ given by
\begin{equation} \label{eq:cubic-matrix}
A_{(j,k),i} = a_{ijk} \quad\mbox{for } i=1,\ldots,n_1; \, j=1,\ldots,n_2; \, k=1,\ldots,n_3.
\end{equation}
Suppose that $A$ has full column rank.  Then, the function $x^1 \mapsto {\sf vec}_p\left( \mathcal{A}(x^1) \right)$ defines a norm on $\R^{n_1}$.
\end{prop}
\begin{proof}
  Using the definition of $\mathcal{A}(x^1)$ and the derivation in the proof of Proposition \ref{prop:grothendieck}, we have
\begin{equation} \label{eq:pqnorm-vec}
\begin{array}{ccc@{\quad}l}
{\sf vec}_p\left( \mathcal{A}(x^1) \right) &=& \mbox{maximize} & \displaystyle{ \sum_{i=1}^{n_1} \left( \sum_{j=1}^{n_2} \sum_{k=1}^{n_3} a_{ijk} u_j^Tv_k \right) x_i^1 } \\
\noalign{\medskip}
& & \mbox{subject to} & \|\mathbf{u}\|_p \le 1, \, \|\mathbf{v}\|_p \le 1, \\
\noalign{\medskip}
& & & \mathbf{u} = (\|u_1\|_2,\ldots,\|u_{n_2}\|_2) \in \R^{n_2}, \\
\noalign{\medskip}
& & & \mathbf{v} = (\|v_1\|_2,\ldots,\|v_{n_3}\|_2) \in \R^{n_3}
\end{array}
\end{equation}
for any $p\in(2,\infty]$.
In particular, ${\sf vec}_p(\mathcal{A}(\cdot))$ is the pointwise supremum of a collection of linear functions, which implies that ${\sf vec}_p(\mathcal{A}(\cdot))$ is convex.  Moreover, it is clear that ${\sf vec}_p\left( \mathcal{A}(kx^1) \right) = |k|\cdot{\sf vec}_p\left( \mathcal{A}(x^1) \right)$ for any $k \in \R$ and $x^1 \in \R^{n_1}$, which together with the convexity of ${\sf vec}_p(\mathcal{A}(\cdot))$ implies that ${\sf vec}_p(\mathcal{A}(\cdot))$ satisfies the triangle inequality.  Finally, let $x^1 \in \R^{n_1} \backslash \{\bz\}$ be arbitrary.  Note that $A$ has full column rank if and only if $\sum_{i=1}^{n_1} a_{ijk}x_i \not= 0$ for some $j=1,\ldots,n_2$ and $k=1,\ldots,n_3$ if and only if
$$ {\sf vec}_p\left( \mathcal{A}(x^1) \right) \ge \max_{1\le j\le n_2,\,1\le k\le n_3} \left| \sum_{i=1}^{n_1} a_{ijk}x_i \right| > 0. $$
This shows that $x^1 = \bz$ whenever ${\sf vec}_p\left( \mathcal{A}(x^1) \right) = 0$, and the proof is completed. 
\end{proof}

\medskip
\noindent Using the argument in~\cite[Section 3.1]{S11a}, we may assume without loss that $A$ has full column rank; i.e., ${\sf vec}_p(\mathcal{A}(\cdot))$ defines a norm on $\R^{n_1}$.  We shall denote this norm by $\|\cdot\|_\mathcal{A}$ in the sequel.

To proceed, consider the unit ball of the norm $\|\cdot\|_\mathcal{A}$ and its polar, which are given by
$$ B_{\mathcal{A}} = \left\{ x \in \R^{n_{1}}: \|x\|_{\mathcal{A}} \leq 1 \right\} $$
and
$$ B_{\mathcal{A}}^{\circ} = \left\{ y\in \R^{n_{1}}: x^Ty \le 1 \mbox{ for all } x \in B_{\mathcal{A}} \right\}, $$
respectively.  Note that both $B_\mathcal{A}$ and $B_\mathcal{A}^\circ$ are centrally symmetric and convex.  Now, using the dual characterization of norms and H\"{o}lder's inequality, we can write Problem (\ref{eq:max-norm-cubic}) as
\begin{equation} \label{eq:norm-diam}
   \max_{\|x\|_p \le 1} \|x\|_\mathcal{A} = \max_{\|x\|_p \le 1} \max_{y \in B_\mathcal{A}^\circ} x^Ty = \max_{y \in B_\mathcal{A}^\circ} \|y\|_q = \frac{1}{2} \mbox{diam}_q(B_\mathcal{A}^\circ),
\end{equation}
where $q=p/(p-1)$ is the conjugate of $p$ and $\mbox{diam}_q(B_\mathcal{A}^\circ)$ is the $L_q$--diameter of $B_\mathcal{A}^\circ$.  In particular, our original problem of approximating $v_{\sf ML}(\mathcal{A},3)$ (see (\ref{eq:trilinear})) is reduced to that of approximating $\mbox{diam}_q(B_\mathcal{A}^\circ)$, which is well studied in the literature.  In the following, we shall present two algorithms for approximating $\mbox{diam}_q(B_\mathcal{A}^\circ)$.  The first is deterministic and implements an idea of Brieden et al.~\cite{BGK+01}.  The second is based on a probabilistic argument of Khot and Naor~\cite{KN08}.  Although the latter is randomized, it is much simpler to implement and achieves a better approximation ratio than the former.

\subsubsection{Approximating the $L_q$--Diameter of $B_\mathcal{A}^\circ$ when $q\in[1,2)$}
\noindent{\bf Deterministic Approximation of $\mbox{diam}_q(B_\mathcal{A}^\circ)$.} The key observation underlying the deterministic approximation algorithm is that the diameter of a convex body with respect to a \emph{polytopal norm} can be computed to arbitrary accuracy in deterministic polynomial time under certain conditions~\cite{BGK+01}.  Thus, in order to approximate the $L_{q}$--diameter of $B_\mathcal{A}^\circ$, it suffices to first construct a centrally symmetric polytope $P$ that approximates the unit $L_{q}$--ball, and then compute the diameter of $B_\mathcal{A}^\circ$ with respect to the polytopal norm induced by $P$.  Before we describe the algorithm in more detail, let us recall some definitions from the algorithmic theory of convex bodies (see~\cite{GLS93} for further details).  For $p>2$, let $B_p^n(r) = \{x \in \R^n: \|x\|_p \le r\}$ denote the $n$--dimensional $L_p$--ball centered at the origin with radius $r>0$.  Let $K$ be a centrally symmetric convex body in $\R^n$. For any $\epsilon \geq 0$, the \emph{outer parallel body} and \emph{inner parallel body} of $K$ are given by
$$ K(\epsilon) = K+B_{2}^{n}(\epsilon) \quad \mbox{and} \quad K(-\epsilon)=\{x\in \mathbb{R}^{n}: x+B_{2}^{n}(\epsilon)\subset K\},$$
respectively. We say that $K$ is \emph{well--bounded} if there exist rational numbers $0<r\leq R < \infty $ such that $B_{2}^{n}(r)\subset K \subset B_{2}^{n}(R)$.  The \emph{weak membership problem} associated with $K$ is defined as follows:

\medskip
\noindent{\sc Weak Membership Problem}. Given a vector $y\in\Q^n$ and a rational number $\epsilon>0$, either (i) assert that $y\in K(\epsilon)$, or (ii) assert that $y\not\in K(-\epsilon)$.

\medskip
\noindent A {\it weak membership oracle} for $K$ is a black box that solves the weak membership problem associated with $K$.  

The starting point of our algorithm for approximating $\mbox{diam}_q(B_\mathcal{A}^\circ)$ is the following result of Brieden et al.~\cite{BGK+01}:
\begin{thm} \label{thm:lq-diam-approx}
   Given an integer $n\ge1$ and a rational number $q \in (1,2]$, one can construct in deterministic polynomial time a centrally symmetric polytope $P$ in $\R^n$ such that (i) $B_q^n(1) \subset P \subset B_q^n\left( O\left( n^{1/2}/(\log n)^{1/p} \right) \right)$, where $p=q/(q-1)$ is the conjugate of $q$, and (ii) for any well--bounded centrally symmetric convex body $K$ in $\R^n$, one has
$$ \Omega\left( \frac{(\log n)^{1/p}}{n^{1/2}} \right)\cdot\mbox{diam}_q(K) \le \mbox{diam}_P(K) \le \mbox{diam}_q(K), $$
where $\mbox{diam}_P(K)$ is the diameter of $K$ with respect to the polytopal norm $\|\cdot\|_P$ induced by $P$ (i.e., for any $x\in\R^n$, one has $\|x\|_P=\min\{\lambda\ge0:x\in\lambda P\}$, and $P$ is the unit ball of the induced norm).  Moreover, if $K$ is equipped with a weak membership oracle, then for any given rational number $\epsilon>0$, the quantity $\mbox{diam}_P(K)$ can be computed to an accuracy of $\epsilon$ in deterministic oracle--polynomial time\footnote{An algorithm has oracle--polynomial time complexity if its runtime is polynomial in both the input size and the number of calls to the oracle~\cite{GLS93}.}, and a vector $x\in K(\epsilon)$ is delivered with $\|x\|_P\ge(1/2) \cdot \mbox{diam}_P(K)-\epsilon$.
\end{thm}
Armed with Theorem \ref{thm:lq-diam-approx}, we see that in order to design a deterministic polynomial--time algorithm for approximating $\mbox{diam}_q(B_\mathcal{A}^\circ)$, it remains to show that $B_\mathcal{A}^\circ$ is well--bounded, and that there is a deterministic polynomial--time algorithm for solving the weak membership problem associated with $B_\mathcal{A}^\circ$.  This is done in the following proposition:
\begin{prop} \label{prop:ellip-prop} 
Let $\mathcal{A}=(a_{ijk}) \in \Q^{n_1\times n_2\times n_3}$ be an arbitrary non--zero order--$3$ tensor, and let $A$ be the $(n_2\times n_3)\times n_1$ matrix given by (\ref{eq:cubic-matrix}).  Suppose that $A$ has full column rank.  Then, the following hold for the centrally symmetric convex body $B_\mathcal{A}^\circ$:
\begin{enumerate}
   \item[\subpb] $B_\mathcal{A}^\circ$ is well--bounded.  Specifically, there exist rational numbers $0 < r\le R<\infty$, whose encoding lengths are polynomially bounded by the input size of Problem (\ref{eq:trilinear}), such that $B_2^{n_1}(r) \subset B_{\mathcal{A}}^\circ\subset B_2^{n_1}(R)$.

   \item[\subpb] The weak membership problem associated with $B_\mathcal{A}^\circ$ can be solved in deterministic polynomial time.
\end{enumerate}
\resetspb
\end{prop}
\begin{proof}
\begin{enumerate}
   \item[\subpb] By polarity, we have $B_2^{n_1}(r) \subset B_{\mathcal{A}}^\circ \subset B_2^{n_1}(R)$ if and only if $B_2^{n_1}(1/R) \subset B_{\mathcal{A}} \subset B_2^{n_1}(1/r)$.  Thus, it suffices to show that $B_\mathcal{A}$ is well--bounded.  Now, using the argument in~\cite[Proposition 2]{S11a} and the assumption that $A$ has full column rank, one can show that $B_2^{n_1}(r') \subset B_\mathcal{A} \subset B_2^{n_1}(R')$, where
$$ r' = \frac{1}{\lceil \sqrt{n_1} \rceil \cdot m}, \quad m = \max_{1 \le i \le n_1} \sum_{j=1}^{n_2}\sum_{k=1}^{n_3} |a_{ijk}| $$
and
$$ R' = \left\lceil \sqrt{\frac{n_2n_3}{\lambda_{min}(A^TA)}} \right\rceil $$
are rational numbers and satisfy $0 < r' \le R'<\infty$.  Moreover, the encoding lengths of $r'$ and $R'$ can be polynomially bounded by the input size of Problem (\ref{eq:trilinear}); see~\cite{GLS93}.  This establishes (a).

   \item[\subpb] By the well--boundedness of $B_\mathcal{A}$ and the results in~\cite[Chapter 4]{GLS93}, it suffices to show that the weak membership problem associated with $B_\mathcal{A}$ can be solved in deterministic polynomial time.  However, this follows directly from the argument in~\cite[Proposition 3]{S11a} and the observation that $\|x\|_\mathcal{A}$ can be computed to arbitrary accuracy in deterministic polynomial time (see (\ref{eq:pqnorm-cvx}) and the remarks following it).
\end{enumerate}
\resetspb
\vspace{-0.5cm}
\end{proof}

\medskip
Using (\ref{eq:norm-diam}), Proposition \ref{prop:ellip-prop} and Theorem \ref{thm:lq-diam-approx}, we conclude that the optimal value of Problem (\ref{eq:max-norm-cubic}) can be approximated to within a factor of $\Omega\left( (\log n_1)^{1/p} \big/ n_1^{1/2} \right)$ in deterministic polynomial time.  Thus, by (\ref{maxmax}) and Proposition \ref{prop:grothendieck}, the optimal value of Problem (\ref{eq:trilinear}) can also be approximated to within a factor of $\Omega\left( (\log n_1)^{1/p} \big/ n_1^{1/2} \right)$ in deterministic polynomial time.  To extract a feasible solution to Problem (\ref{eq:trilinear}) with the stated approximation guarantee, we just need to unwind our sequence of reductions.  For simplicity, let us assume that all computations can be done exactly.  Then, by Proposition \ref{prop:ellip-prop} and Theorem \ref{thm:lq-diam-approx}, we can find a centrally symmetric polytope $P$ and a vector $\bar{y} \in B_\mathcal{A}^\circ$ such that
\begin{equation} \label{eq:qq}
   \|\bar{y}\|_q \ge \|\bar{y}\|_P = \frac{1}{2}\mbox{diam}_P(B_\mathcal{A}^\circ) \ge \Omega\left( \frac{(\log n_1)^{1/p}}{n_1^{1/2}} \right) \cdot \mbox{diam}_q(B_\mathcal{A}^\circ).
\end{equation}
Now, define the vector $\bar{x}^1 \in \R^{n_1}$ by
$$ \bar{x}^1_i = \frac{\mbox{sgn}(\bar{y}_i) \cdot |\bar{y}_i|^{q-1}}{\|\bar{y}\|_q^{q-1}} \quad\mbox{for } i=1,\ldots,n_1. $$
It is easy to verify that $\|\bar{x}^1\|_p = 1$ and
\begin{equation} \label{eq:cvx-q}
   {\sf vec}_p\left( \mathcal{A}(\bar{x}^1) \right) = \|\bar{x}^1\|_\mathcal{A} = (\bar{x}^1)^T\bar{y} = \|\bar{y}\|_q. 
\end{equation}
In particular, by applying the deterministic version of Algorithm \ref{alg:pq-norm} to the $n_2\times n_3$ matrix $\mathcal{A}(\bar{x}^1)$, we can extract two vectors $\bar{x}^2 \in \R^{n_2}$ and $\bar{x}^3 \in \R^{n_3}$ such that $\|\bar{x}^2\|_p = \|\bar{x}^3\|_p=1$ and
\begin{equation} \label{eq:tri-vec}
   \sum_{i=1}^{n_1}\sum_{j=1}^{n_2}\sum_{k=1}^{n_3} a_{ijk}\bar{x}_i^1\bar{x}_j^2\bar{x}_k^3 \ge \frac{1}{27} {\sf vec}_p\left( \mathcal{A}(\bar{x}^1) \right). 
\end{equation}
Finally, since (\ref{maxmax}), (\ref{eq:norm-diam}) and Proposition \ref{prop:grothendieck} together imply
$$ \frac{1}{2}\mbox{diam}_q(B_\mathcal{A}^\circ) = \max_{\|x\|_p \le 1} \|x\|_\mathcal{A} \ge \max_{\|x^1\|_p \le 1} \|\mathcal{A}(x^1)\|_{p \limto q} = v_{\sf ML}(\mathcal{A},3), $$
we conclude from (\ref{eq:qq})--(\ref{eq:tri-vec}) that $(\bar{x}^1,\bar{x}^2,\bar{x}^3) \in \R^{n_1} \times \R^{n_2} \times \R^{n_3}$ is an $\Omega\left( (\log n_1)^{1/p} \big/ n_1^{1/2} \right)$--approximate solution to Problem (\ref{eq:trilinear}). 

Recall that the above conclusion is obtained under the assumption that all computations are exact.  However, it can be shown via a similar but more tedious calculation that the same conclusion holds when the computations are inexact; cf.~\cite{S11a}.  Thus, we have proven the following theorem:
\begin{thm} \label{thm:trilinear}
For any given $p\in(2,\infty]$, there is a deterministic polynomial--time approximation algorithm for Problem (\ref{eq:trilinear}) with approximation ratio $\Omega\left( (\log n_1)^{1/p} \big/ n_1^{1/2} \right)$.
\end{thm}
The following corollary is a direct consequence of Theorems \ref{hpml} and \ref{thm:trilinear}:
\begin{coro}
For $d=3$ and any given $p\in(2,\infty]$, there is a deterministic polynomial--time approximation algorithm for Problem $({\sf HP})$ with approximation ratio $\Omega\left( (\log n)^{1/p} / n^{1/2} \right)$.
\end{coro}

\medskip
\noindent {\bf Randomized Approximation of $\mbox{diam}_q(B_\mathcal{A}^\circ)$.} In this section, we consider an alternative approach to approximating $\mbox{diam}_q(B_\mathcal{A}^\circ)$, namely, via randomization.  The theoretical underpinning of this approach is the following probabilistic results due to Khot and Naor~\cite{KN08}:
\begin{prop} \label{prop:prob}
The following hold:
\begin{enumerate}
\item[\subpb] Let $\zeta_1,\ldots,\zeta_n$ be i.i.d.~Bernoulli random variables and set $\zeta=(\zeta_1,\ldots,\zeta_n) \in \R^n$.  Then, there exist universal constants $\delta_0,c_0 > 0$ such that for every $w \in \R^n$,
$$ \Pr \left( w^T\zeta \ge \sqrt{\frac{\delta_0 \log n}{n}} \cdot \|w\|_1 \right) \ge \frac{c_0}{n^{\delta_0}}. $$

\item[\subpb] Suppose that $q\in(1,2)$, and let $p=q/(q-1)$.  Let $\xi_1,\ldots,\xi_n$ be i.i.d.~random variables with density $p\cdot\exp(-|t|^p)/(2\Gamma(1/p))$ and set $\xi=(\xi_1,\ldots,\xi_n) \in \R^n$.  Then, there exist universal constants $\delta_1,c_1,c_2,\bar{n}>0$ such that for all $n \ge \bar{n}$, we have 
$$ \Pr\left( \frac{w^T\xi}{\|\xi\|_p} \ge \sqrt{\frac{\delta_1 \log n}{n}} \cdot \|w\|_q \right) \ge \frac{c_1}{n^{c_2}} $$ 
for every $w \in \R^n$.
\end{enumerate}
\resetspb
\end{prop}
{\bf Remark.} An inspection of the proofs in~\cite{KN08} reveals that one can take
\begin{eqnarray}
& \displaystyle{\delta_0 = \frac{1}{48}, \quad c_0 = \frac{1}{72}}, & \label{eq:inf-norm-const} \\
\noalign{\medskip}
& \displaystyle{ \delta_1 = \frac{\E{\xi_1^2}}{160 \times 2^{2/q}} > \frac{3}{6400}, \quad c_1 = \frac{1}{144}, \quad c_2 = \frac{1}{40}, \quad \bar{n} = 41. } & \label{eq:q-norm-const}
\end{eqnarray}

Using Proposition \ref{prop:prob}, we can prove the following result:
\begin{prop} \label{prop:rand-diam-approx}
   For any given $q\in[1,2)$, there is a randomized polynomial--time algorithm that returns a vector $v\in\R^{n_1}$ with the following property:
$$ \Pr\left[ \Omega\left( \sqrt{\frac{\log n_1}{n_1}} \right) \cdot \mbox{diam}_q(B_\mathcal{A}^\circ) \le 2\|v\|_\mathcal{A} \le \mbox{diam}_q(B_\mathcal{A}^\circ) \right] \ge \frac{1}{2}. $$
\end{prop}
\begin{proof}
Since $B_\mathcal{A}^\circ$ is compact and $x \mapsto \|x\|_q$ is continuous, there exists a $\bar{y} \in B_\mathcal{A}^\circ$ such that $\|\bar{y}\|_q = \mbox{diam}_q(B_\mathcal{A}^\circ)/2$.  We consider two cases:

\medskip
\noindent{\bf Case 1: $q=1$.}  Let $\delta_0,c_0$ be as in (\ref{eq:inf-norm-const}) and set $M=(\ln 2)n_1^{\delta_0}/c_0$.  Consider a collection $\{\zeta_j^i:i=1,\ldots,M;\, j=1,\ldots,{n_1}\}$ of i.i.d.~Bernoulli random variables.  Define
\begin{eqnarray*}
 & \displaystyle{ \zeta^i=(\zeta_1^i,\ldots,\zeta_{n_1}^i) \in \R^{n_1} \quad \mbox{for } i=1,\ldots,M, } & \\
\noalign{\medskip}
 & \displaystyle{ i^* = \arg\max_{1\le i\le M} \|\zeta^i\|_\mathcal{A}, \quad v = \zeta^{i^*}, \quad \tau = 2\|v\|_\mathcal{A}. } &
\end{eqnarray*}
We claim that $v$ has the desired property.  Indeed, it is clear from (\ref{eq:norm-diam}) that $\tau \le \mbox{diam}_1(B_\mathcal{A}^\circ)$.  Moreover, upon recalling that $\|\zeta^i\|_\mathcal{A} = \max_{y \in B_\mathcal{A}^\circ} y^T\zeta^i$ and using Proposition \ref{prop:prob}(a), we have
\begin{eqnarray*}
\Pr \left ( \tau \ge \sqrt{\frac{\delta_0 \log n_{1}}{n_{1}}}\cdot \mbox{diam}_1(B_{\mathcal{A}}^{\circ}) \right) &\ge& 1 - \Pr \left( \bigcap_{i=1}^M \left\{ \bar{y}^T\zeta^i < \sqrt{\frac{ \delta_0 \log n_{1}}{n_{1}}}\cdot \|\bar{y}\|_1 \right\}\right)\\
\noalign{\medskip}
& \geq & 1 - \left( 1 - \frac{c_0}{n_1^{\delta_0}} \right)^M \\
\noalign{\medskip}
&\ge& \frac{1}{2},
\end{eqnarray*}
which establishes the claim.

\medskip
\noindent{\bf Case 2: $q\in(1,2)$.}  Let $\delta_1,c_1,c_2$ be as in (\ref{eq:q-norm-const}) and set $M=(\ln 2)n_1^{c_2}/c_1$.  Consider a collection $\{\xi_j^i:i=1,\ldots,M;\, j=1,\ldots,{n_1}\}$ of i.i.d.~random variables with density $p\cdot\exp(-|t|^p)/(2\Gamma(1/p))$, where $p=q/(q-1)$.  Define
\begin{eqnarray*}
 & \displaystyle{ \xi^i=(\xi_1^i,\ldots,\xi_{n_1}^i) \in \R^{n_1}, \quad \bar{\xi}^i = \frac{\xi^i}{\|\xi^i\|_p} \quad \mbox{for } i=1,\ldots,M, } & \\
\noalign{\medskip}
& \displaystyle{ i^* = \arg\max_{1\le i\le M} \|\bar{\xi}^i\|_\mathcal{A}, \quad v = \bar{\xi}^{i^*}, \quad \tau = 2\|v\|_\mathcal{A}. } & 
\end{eqnarray*}
Using Proposition \ref{prop:prob}(b) and our previous argument, we have $\tau \le \mbox{diam}_q(B_\mathcal{A}^\circ)$ and
\begin{eqnarray*}
\Pr \left ( \tau \ge \sqrt{\frac{ \delta_1 \log n_{1}}{n_{1}}}\cdot \mbox{diam}_{q}(B_{\mathcal{A}}^{\circ}) \right) &\ge& 1 - \Pr \left( \bigcap_{i=1}^M \left\{ \bar{y}^T\bar{\xi}^i < \sqrt{\frac{ \delta_1 \log n_{1}}{n_{1}}}\cdot \|\bar{y}\|_{q} \right\}\right)\\
\noalign{\medskip}
& \geq & 1 - \left( 1 - \frac{c_{1}}{n_{1}^{c_{2}}} \right)^M \\
\noalign{\medskip}
&\ge& \frac{1}{2}.
\end{eqnarray*}
This completes the proof of Proposition \ref{prop:rand-diam-approx}.
\end{proof}

\medskip
\noindent By combining Proposition \ref{prop:rand-diam-approx} with the procedure outlined in the paragraph above Theorem \ref{thm:trilinear}, we can extract an $\Omega(\sqrt{\log n_1/n_1})$--approximate solution to Problem (\ref{eq:trilinear}).  Thus, we have proven the following theorem:
\begin{thm}
For any given $p\in(2,\infty]$, there is a randomized polynomial--time approximation algorithm for Problem (\ref{eq:trilinear}) with approximation ratio $\Omega(\sqrt{\log n_{1}/ n_{1}})$.  In particular, for $d=3$ and any given $p\in(2,\infty]$, there is a randomized polynomial--time approximation algorithm for Problem $({\sf HP})$ with approximation ratio $\Omega(\sqrt{\log n/ n})$.
\end{thm}

\subsection{General Case: Approximating $L_p$--Ball Constrained Multilinear Maximization via Recursion}

Now, let us consider the problem of maximizing a degree--$d$ multilinear form over $L_p$--balls, where $d\ge4$ and $p\in(2,\infty]$ are fixed.  Our approach is based on the following simple observation:  Let $\mathcal{A} \in \R^{n_{1}\times n_{2}\times \cdots \times n_{d}}$ be an arbitrary non--zero order--$d$ tensor.  Then,
$$ v_{\sf ML}(\mathcal{A},d) = \max_{\|x^1\|_p\le1} v_{\sf ML}(\mathcal{A}(x^1),d-1). $$
This suggests that it may be possible to approximate the degree--$d$ problem $v_{\sf ML}(\mathcal{A},d)$ if we have an algorithm for approximating the degree--$(d-1)$ problem $v_{\sf ML}(\mathcal{B},d-1)$, where $\mathcal{B}$ is an arbitrary non--zero order--$(d-1)$ tensor.  To implement this idea, we proceed as follows.  Let $\mathcal{H}$ be an arbitrary Hilbert space.  Given an arbitrary non--zero order--$d$ tensor $\mathcal{A} = (a_{i_1i_2\cdots i_d}) \in \R^{n_1 \times \cdots \times n_d}$, let $F_\mathcal{A}$ be the associated multilinear form, and define a function $\tilde{F}_\mathcal{A}:\R^{n_1}\times\cdots\times\R^{n_{d-2}}\times \mathcal{H}^{n_{d-1}} \times \mathcal{H}^{n_d} \limto \R$ by
$$ \tilde{F}_\mathcal{A}\left( x^1,\ldots,x^{d-2},\{u_j\}_{j=1}^{n_{d-1}},\{v_k\}_{k=1}^{n_d} \right) = \sum_{i_1=1}^{n_1}\cdots\sum_{i_{d-2}=1}^{n_{d-2}}\sum_{j=1}^{n_{d-1}}\sum_{k=1}^{n_d} a_{i_1\cdots i_{d-2}jk} \cdot x_{i_1}^1\cdots x_{i_{d-2}}^{d-2} \cdot u_j^Tv_k. $$
By Proposition \ref{prop:grothendieck}, for any given $\bar{x}^i \in \R^{n_i}$, where $i=1,\ldots,d-2$, we have
$$ \frac{1}{K_G} \cdot {\sf vec}_p\left( \mathcal{A}(\bar{x}^1,\ldots,\bar{x}^{d-2}) \right) \le \left\| \mathcal{A}(\bar{x}^1,\ldots,\bar{x}^{d-2}) \right\|_{p\limto q} \le {\sf vec}_p\left( \mathcal{A}(\bar{x}^1,\ldots,\bar{x}^{d-2}) \right). $$
Since
$$ \left\| \mathcal{A}(\bar{x}^1,\ldots,\bar{x}^{d-2}) \right\|_{p\limto q} = \max\left\{ F_\mathcal{A}(\bar{x}^1,\ldots,\bar{x}^{d-2},x^{d-1},x^d) : \|x^{d-1}\|_p \le 1, \, \|x^d\|_p \le 1 \right\}, $$
it follows that
$$ \frac{1}{K_G} \cdot r_{\sf ML}(\mathcal{A},d) \le \max_{\|x^i\|_p \le 1, \,i=1,\ldots,d}  F_\mathcal{A}(x^1,\ldots,x^d) = v_{\sf ML}(\mathcal{A},d) \le r_{\sf ML}(\mathcal{A},d), $$
where
\begin{equation} \label{eq:ML-relax}
\begin{array}{ccc@{\quad}l}
   r_{\sf ML}(\mathcal{A},d) &=& \mbox{maximize} & \displaystyle{ \tilde{F}_\mathcal{A}\left( x^1,\ldots,x^{d-2},\{u_j\}_{j=1}^{n_{d-1}},\{v_k\}_{k=1}^{n_d} \right) } \\
   \noalign{\medskip}
   & & \mbox{subject to} & \|x^i\|_p\le 1 \quad \mbox{for } i=1,\ldots,d-2, \\
   \noalign{\medskip}
   & & & \|\mathbf{u}\|_p \le 1, \, \|\mathbf{v}\|_p \le 1, \\
   \noalign{\medskip}
   & & & \mathbf{u}=(\|u_1\|_2,\ldots,\|u_{n_{d-1}}\|_2) \in \R^{n_{d-1}}, \\
   \noalign{\medskip}
   & & & \mathbf{v}=(\|v_1\|_2,\ldots,\|v_{n_d}\|_2) \in \R^{n_d}.
\end{array}
\end{equation}
In particular, $v_{\sf ML}(\mathcal{A},d)$ and $r_{\sf ML}(\mathcal{A},d)$ are equivalent from the approximation perspective.  In the sequel, we shall focus on designing approximation algorithms for the latter using both deterministic and randomized approaches.

\subsubsection{Deterministic Approximation of $r_{\sf ML}(\mathcal{A},d)$}
Our deterministic approach is motivated by the results developed in~\cite{S11a}.  Before delving into the details, let us give an overview of the approach.  Suppose there is a deterministic algorithm that can approximate the problem $r_{\sf ML}(\mathcal{B},d-1)$ for any non--zero order--$(d-1)$ tensor $\mathcal{B}$, where $d\ge4$ is fixed.  Then, given an arbitrary $x^1 \in \R^{n_1}$, since $\mathcal{A}(x^1)$ is an order--$(d-1)$ tensor, we can apply the algorithm to the problem
$$
\begin{array}{ccc@{\quad}l}
   r_{\sf ML}(\mathcal{A}(x^1),d-1) &=& \mbox{maximize} & \displaystyle{ \tilde{F}_\mathcal{A}\left( x^1,\ldots,x^{d-2},\{u_j\}_{j=1}^{n_{d-1}},\{v_k\}_{k=1}^{n_d} \right) } \\
   \noalign{\medskip}
   & & \mbox{subject to} & \|x^i\|_p\le 1 \quad \mbox{for } i=2,\ldots,d-2, \\
   \noalign{\medskip}
   & & & \|\mathbf{u}\|_p \le 1, \, \|\mathbf{v}\|_p \le 1, \\
   \noalign{\medskip}
   & & & \mathbf{u}=(\|u_1\|_2,\ldots,\|u_{n_{d-1}}\|_2) \in \R^{n_{d-1}}, \\
   \noalign{\medskip}
   & & & \mathbf{v}=(\|v_1\|_2,\ldots,\|v_{n_d}\|_2) \in \R^{n_d}
\end{array}
$$
and obtain a value $G_{d-1}(x^1)$ that satisfies $\beta_{d-1} \cdot r_{\sf ML}(\mathcal{A}(x^1),d-1) \le G_{d-1}(x^1) \le r_{\sf ML}(\mathcal{A}(x^1),d-1)$, where $\beta_{d-1} \in (0,1)$ is the approximation ratio of the algorithm.  Since this holds for any $x^1 \in \R^{n_1}$, it follows that
$$ \beta_{d-1} \cdot r_{\sf ML}(\mathcal{A},d) \le \max_{\|x^1\|_p \le 1} G_{d-1}(x^1) \le r_{\sf ML}(\mathcal{A},d). $$
Now, if we can show that the function $G_{d-1}$ defines a norm on $\R^{n_1}$, then $\max_{\|x^1\|_p\le1} G_{d-1}(x^1)$ is a norm maximization problem, which can be approximated using the techniques outlined in Section \ref{sec:base}.  This would then yield an approximation algorithm for the problem $r_{\sf ML}(\mathcal{A},d)$.

To carry out this plan, we need the following result:
\begin{prop}\label{prop1}
Let $d \ge 3$ and $p \in (2,\infty]$ be given.  For $i=1,\ldots,d-3$, let $P_i$ be a centrally symmetric polytope in $\R^{n_{i+1}}$ satisfying the properties stated in Theorem \ref{thm:lq-diam-approx}.  Furthermore, let $\mathcal{A}=(a_{i_{1}i_{2}\cdots i_{d}}) \in \R^{ n_{1}\times n_{2}\times \cdots \times n_{d}}$ be an arbitrary non--zero order--$d$ tensor.  Define the functions $\Lambda^{\mathcal{A},d}_{i}: \R^{n_1} \times \R^{n_2} \times \cdots \times \R^{n_i} \limto \R_+$ for $i=1,\ldots,d-2$ inductively as follows:
\begin{equation} \label{eq:recursive-norm}
\begin{array}{rcl}
\displaystyle{ \Lambda^{\mathcal{A},d}_{d-2}(x^{1},x^{2},\ldots,x^{d-2}) } &=&  \displaystyle{ {\sf vec}_p\left( \mathcal{A}(x^{1},x^{2},\ldots,x^{d-2}) \right), } \\
\noalign{\medskip}
\displaystyle{ \Lambda^{\mathcal{A},d}_{i}(x^{1},x^{2},\ldots,x^{i}) } &=& \displaystyle{ \mbox{diam}_{P_i} \left[ \left\{ y \in \R^{n_{i+1}}:\Lambda^{\mathcal{A},d}_{i+1}(x^{1},x^{2},\ldots,x^{i},y)\leq 1 \right\}^{\circ} \right] }
\end{array}
\end{equation}
for $i=d-3,d-4,\ldots,1$.  Then, the following hold:
\begin{enumerate}
\item[\subpb] For $j=1,2,\ldots,d-2$ and for any $\bar{x}^1,\ldots,\bar{x}^{k-1},\bar{x}^{k+1},\ldots,\bar{x}^j$, where $\bar{x}^i \in \R^{n_i}$, the function $\bar{\Lambda}^{\mathcal{A},d}_{j,k}: \R^{n_{k}} \limto \R_{+}$ given by
$$\bar{\Lambda}^{\mathcal{A},d}_{j,k}(x)=\Lambda^{\mathcal{A},d}_{j}(\bar{x}^{1},\ldots,\bar{x}^{k-1},x,\bar{x}^{k+1},\ldots,\bar{x}^{j})$$
is a semi--norm on $\R^{n_{k}}$ for any $k\in \{1,\ldots,j\}$.

\item[\subpb] Let $A$ be the $(n_{2}\times \cdots \times n_{d})\times n_{1}$ matrix given by
\begin{equation} \label{eq:flatten}
A_{(i_{2},\ldots,i_{d}),i_{1}} = a_{i_{1}i_{2}\cdots i_{d}} \quad\mbox{for } i_j = 1,\ldots,n_j; \, j=1,\ldots,d.
\end{equation}
Suppose that $A$ has full column rank.  Then, the function $\Lambda^{\mathcal{A},d}_{1}$ defines a norm on $\R^{n_{1}}$.

\item[\subpb] We have 
$$ \Lambda_{i-1}^{\mathcal{A}(x^1),d-1}(x^2,x^3,\ldots,x^i) = \Lambda_i^{\mathcal{A},d}(x^1,x^2,\ldots,x^i) \quad\mbox{for } i=2,3,\ldots,d-2. $$
\end{enumerate}
\resetspb
\end{prop}
\begin{proof}
Both (a) and (b) are essentially adaptations of the corresponding claims in~\cite[Proposition 4]{S11a}.  To prove (c), we proceed by backward induction on $i$.  For $i=d-2$, we have, by definition,
\begin{eqnarray*}
\Lambda_{d-3}^{\mathcal{A}(x^1),d-1}(x^2,x^3,\ldots,x^{d-2}) &=& {\sf vec}_p\left( [\mathcal{A}(x^1)](x^2,x^3,\ldots,x^{d-2}) \right) \\
\noalign{\medskip}
&=& {\sf vec}_p\left( \mathcal{A}(x^1,x^2,\ldots,x^{d-2}) \right) \\
\noalign{\medskip}
&=& \Lambda_{d-2}^{\mathcal{A},d}(x^1,x^2,\ldots,x^{d-2}).
\end{eqnarray*}
For the inductive step, we use both the definition in (\ref{eq:recursive-norm}) and the inductive hypothesis to obtain
\begin{eqnarray*}
 \Lambda^{\mathcal{A}(x^1),d-1}_{i-1}(x^{2},x^3,\ldots,x^{i}) &=& \mbox{diam}_{P_i} \left[ \left\{ y \in \R^{n_{i+1}}:\Lambda^{\mathcal{A}(x^1),d-1}_i(x^{2},x^3,\ldots,x^{i},y)\leq 1 \right\}^{\circ} \right] \\
\noalign{\medskip}
&=& \mbox{diam}_{P_i} \left[ \left\{ y \in \R^{n_{i+1}}:\Lambda^{\mathcal{A},d}_{i+1}(x^1,x^2,\ldots,x^{i},y)\leq 1 \right\}^{\circ} \right] \\
\noalign{\medskip}
&=& \Lambda_i^{\mathcal{A},d}(x^1,x^2,\ldots,x^i).
\end{eqnarray*}
This completes the proof.
\end{proof}

\medskip
We are now ready to prove the main result of this section:
\begin{thm}\label{thm3}
Let $d\geq 3$ and $p \in (2,\infty]$ be given.  Let $\mathcal{A}=(a_{i_{1}i_{2}\cdots i_d})\in \R^{n_{1}\times n_{2}\times \cdots \times n_{d}}$ be an arbitrary non--zero order--$d$ tensor.  Consider the functions $\{\Lambda^{\mathcal{A},d}_{i}\}^{d-2}_{i=1}$ defined in (\ref{eq:recursive-norm}) and the $(n_{2}\times \cdots \times n_{d})\times n_{1}$ matrix $A$ defined in (\ref{eq:flatten}).  Suppose that $A$ has full column rank.  Then, the following hold:
\begin{enumerate}
\item[\subpb] For any given $x \in \R^{n_1}$, the norm $\Lambda_1^{\mathcal{A},d}(x)$ is \emph{efficiently computable}; i.e., it can be computed to any desired accuracy by a deterministic algorithm whose runtime is polynomial in the input size of Problem (\ref{eq:ML-relax}) and the level of accuracy.

\item[\subpb] There exist rational numbers $0 < r \le R < \infty$, whose encoding lengths are polynomially bounded by the input size of Problem (\ref{eq:ML-relax}), such that 
$$ B_2^{n_1}(r) \subset \left\{ x \in \R^{n_1} : \Lambda_1^{\mathcal{A},d}(x) \le 1 \right\} \subset B_2^{n_1}(R). $$
Consequently, the quantity $\mbox{diam}_{P_0}\left( \left\{ x \in \R^{n_1} : \Lambda_1^{\mathcal{A},d}(x) \le 1 \right\}^\circ \right)$ can be efficiently computed, where $P_0$ is a centrally symmetric polytope in $\R^{n_1}$ satisfying the properties stated in Theorem \ref{thm:lq-diam-approx}. 

\item[\subpb] We have
$$
\Omega\left( \prod_{i=1}^{d-2} \frac{(\log n_i)^{1/p}}{n_i^{1/2}} \right) \cdot r_{\sf ML}(\mathcal{A},d) \le \frac{1}{2}\mbox{diam}_{P_0}\left( \left\{ x \in \R^{n_1} : \Lambda_1^{\mathcal{A},d}(x) \le 1 \right\}^\circ \right) \le r_{\sf ML}(\mathcal{A},d).
$$
\end{enumerate}
\resetspb
In particular, there is a deterministic polynomial--time algorithm for Problem $({\sf ML})$ with approximation ratio $\Omega\left( \prod_{i=1}^{d-2} (\log n_i)^{1/p} \big/ n_i^{1/2} \right)$.
\end{thm}
\begin{proof} 
We proceed by induction on $d\ge 3$. The base case follows from (\ref{eq:norm-diam}), Proposition \ref{prop:ellip-prop} and Theorem \ref{thm:lq-diam-approx}. Now, suppose that $d>3$. Let $x_{1}\in \R^{n_{1}} \backslash \{\mathbf{0}\}$ be arbitrary, and consider the order--$(d-1)$ tensor $\mathcal{A}(x^{1})\in \R^{ n_{2}\times n_{3}\times \cdots \times n_{d}}$.  Without loss of generality, we may assume that the $(n_{3}\times \cdots \times n_{d})\times n_{2}$ matrix $A(x^{1})$, where $[A(x^{1})]_{(i_{3},\ldots,i_{d}),i_{2}}=[\mathcal{A}(x^{1})]_{i_{2}i_{3} \cdots i_{d}}$, has full column rank.  By the inductive hypothesis, $\Lambda_1^{\mathcal{A}(x^1),d-1}$ is an efficiently computable norm on $\R^{n_2}$ and the set $\left\{ x \in \R^{n_2} : \Lambda_1^{\mathcal{A}(x^1),d-1}(x) \le 1 \right\}$ is well--bounded.  Moreover, using (\ref{eq:recursive-norm}) and Proposition \ref{prop1}(c), we have
$$ \Lambda^{\mathcal{A},d}_1(x^{1}) = \mbox{diam}_{P_1} \left[ \left\{ x \in \R^{n_2}:\Lambda^{\mathcal{A},d}_2(x^{1},x)\leq 1 \right\}^{\circ} \right] = \mbox{diam}_{P_1} \left[ \left\{ x \in \R^{n_2}:\Lambda^{\mathcal{A}(x^1),d-1}_1(x)\leq 1 \right\}^{\circ} \right]. $$
Hence, by arguing as in the proof of Proposition \ref{prop:ellip-prop} and applying Theorem \ref{thm:lq-diam-approx}, we conclude that $\Lambda_1^{\mathcal{A},d}$ is an efficiently computable norm on $\R^{n_1}$.

Let $B_{\Lambda_1^{\mathcal{A},d}} = \left\{ x \in \R^{n_1} : \Lambda_1^{\mathcal{A},d}(x) \le 1 \right\}$ be the unit ball of $\Lambda_1^{\mathcal{A},d}$.  Using the argument in the proof of~\cite[Theorem 4]{S11a}, one can show that $B_{\Lambda_1^{\mathcal{A},d}}$ is well--bounded.  As a corollary, we see that $B_{\Lambda_1^{\mathcal{A},d}}^\circ$ is also well--bounded, and that the weak membership problem associated with $B_{\Lambda_1^{\mathcal{A},d}}^\circ$ can be solved in deterministic polynomial time.  This implies that $\mbox{diam}_{P_0}\left( \left\{ x \in \R^{n_1} : \Lambda_1^{\mathcal{A},d}(x) \le 1 \right\}^\circ \right)$ can be efficiently computed.

Now, the inductive hypothesis, the definition of $\Lambda_1^{\mathcal{A},d}$ in (\ref{eq:recursive-norm}) and Proposition \ref{prop1}(c) yield
\begin{eqnarray*}
\Omega\left( \prod_{i=2}^{d-2} \frac{(\log n_i)^{1/p}}{n_i^{1/2}} \right) \cdot r_{\sf ML}(\mathcal{A}(x^1),d-1) &\le& \frac{1}{2}\mbox{diam}_{P_1}\left( \left\{ x \in \R^{n_2} : \Lambda_1^{\mathcal{A}(x^1),d-1}(x) \le 1 \right\}^\circ \right) \\
\noalign{\medskip}
&=& \frac{1}{2}\Lambda_1^{\mathcal{A},d}(x^1) \\
\noalign{\medskip}
&\le& r_{\sf ML}(\mathcal{A}(x^1),d-1).
\end{eqnarray*}
Since $r_{\sf ML}(\mathcal{A},d) = \max_{\|x^1\|_p \le 1} r_{\sf ML}(\mathcal{A}(x^1),d)$, it follows that
\begin{equation} \label{eq:max-diam-1}
\Omega\left( \prod_{i=2}^{d-2} \frac{(\log n_i)^{1/p}}{n_i^{1/2}} \right) \cdot r_{\sf ML}(\mathcal{A},d) \le \frac{1}{2} \max_{\|x^1\|_p \le 1} \Lambda_1^{\mathcal{A},d}(x^1) \le r_{\sf ML}(\mathcal{A},d).
\end{equation}
By mimicking the derivation of (\ref{eq:norm-diam}), one can show that
\begin{equation} \label{eq:max-diam-2}
\max_{\|x^1\|_p \le 1} \Lambda_1^{\mathcal{A},d}(x^1) = \frac{1}{2} \mbox{diam}_q\left( B_{\Lambda_1^{\mathcal{A},d}}^\circ \right).
\end{equation}
Moreover, since $B_{\Lambda_1^{\mathcal{A},d}}^\circ$ is well--bounded, Theorem \ref{thm:lq-diam-approx} and the definition of $P_0$ imply that
\begin{equation} \label{eq:max-diam-3}
\Omega\left( \frac{(\log n_1)^{1/p}}{n_1^{1/2}} \right) \cdot \mbox{diam}_q\left( B_{\Lambda_1^{\mathcal{A},d}}^\circ \right) \le \mbox{diam}_{P_0}\left( B_{\Lambda_1^{\mathcal{A},d}}^\circ \right) \le \mbox{diam}_q\left( B_{\Lambda_1^{\mathcal{A},d}}^\circ \right).
\end{equation}
It then follows from (\ref{eq:max-diam-1}), (\ref{eq:max-diam-2}) and (\ref{eq:max-diam-3}) that
$$ \Omega\left( \prod_{i=1}^{d-2} \frac{(\log n_i)^{1/p}}{n_i^{1/2}} \right) \cdot r_{\sf ML}(\mathcal{A},d) \le \frac{1}{2} \mbox{diam}_{P_0}\left( \left\{ x \in \R^{n_1} : \Lambda_1^{\mathcal{A},d}(x) \le 1 \right\}^\circ \right) \le r_{\sf ML}(\mathcal{A},d). $$
This completes the proof of Theorem \ref{thm3}.
\end{proof}

\medskip
The following is an immediate corollary of Theorems \ref{hpml} and \ref{thm3}:
\begin{coro}
For any given $d\ge3$ and $p\in(2,\infty]$, there is a deterministic polynomial--time algorithm for $({\sf HP})$ with approximation ratio (resp.~relative approximation ratio) $\Omega\left( (\log n)^{(d-2)/p} \big/ n^{d/2-1} \right)$ when $d \ge 3$ is odd (resp.~even).
\end{coro}

\subsubsection{Randomized Approximation of $r_{\sf ML}(\mathcal{A},d)$}

As in the case where $d=3$, we can approximate $r_{\sf ML}(\mathcal{A},d)$ using a randomized approach.  Such an approach is based on the following result, which states that every optimal solution to Problem (\ref{eq:ML-relax}) satisfies certain probabilistic inequality:
\begin{prop} \label{prop:prob-high-d}
Let 
$$ \left( \bar{x}^1,\ldots,\bar{x}^{d-2},\{\bar{u}_j\}_{j=1}^{n_{d-1}},\{\bar{v}_k\}_{k=1}^{n_d} \right) $$
be an optimal solution to Problem (\ref{eq:ML-relax}).  
\begin{enumerate}
\item[\subpb] Let $\zeta \in \R^{n_1}$ be a vector of i.i.d.~Bernoulli random variables.  Then,
$$ \Pr\left( \tilde{F}_\mathcal{A}\left( \zeta, \bar{x}^2, \ldots, \bar{x}^{d-2}, \{\bar{u}_j\}_{j=1}^{n_{d-1}},\{\bar{v}_k\}_{k=1}^{n_d} \right) \ge \sqrt{\frac{\delta_0 \log n_1}{n_1}} \cdot r_{\sf ML}(\mathcal{A},d) \right) \ge \frac{c_0}{n_1^{\delta_0}}, $$
where the constants $\delta_0,c_0$ are given by (\ref{eq:inf-norm-const}).

\item[\subpb] Let $\xi \in \R^{n_1}$ be a vector of i.i.d.~random variables with density $p\cdot\exp(-|t|^p)/(2\Gamma(1/p))$, and set $\bar{\xi} = \xi/\|\xi\|_p$.  Then, 
$$ \Pr\left( \tilde{F}_\mathcal{A}\left( \bar{\xi}, \bar{x}^2, \ldots, \bar{x}^{d-2}, \{\bar{u}_j\}_{j=1}^{n_{d-1}},\{\bar{v}_k\}_{k=1}^{n_d} \right) \ge \sqrt{\frac{\delta_1 \log n_1}{n_1}} \cdot r_{\sf ML}(\mathcal{A},d) \right) \ge \frac{c_1}{n_1^{c_2}} $$
for all $n \ge \bar{n}$, where the constants $\delta_1,c_1,c_2,\bar{n}$ are given by (\ref{eq:q-norm-const}).
\end{enumerate}
\resetspb
\end{prop}
\begin{proof}
   Let $w \in \R^{n_1}$ be the vector defined by
$$ w_{i_1} = \sum_{i_2=1}^{n_2}\cdots\sum_{i_{d-2}=1}^{n_{d-2}}\sum_{j=1}^{n_{d-1}}\sum_{k=1}^{n_d} a_{i_1 i_2 \cdots i_{d-2}jk} \cdot \bar{x}_{i_2}^2\cdots \bar{x}_{i_{d-2}}^{d-2} \cdot \bar{u}_j^T\bar{v}_k \quad\mbox{for } i_1=1,\ldots,n_1. $$
Then, for any $x \in \R^{n_1}$, we have
$$ w^Tx = \tilde{F}_\mathcal{A}\left( x, \bar{x}^2, \ldots, \bar{x}^{d-2}, \{\bar{u}_j\}_{j=1}^{n_{d-1}},\{\bar{v}_k\}_{k=1}^{n_d} \right). $$
Moreover, by the definition of $r_{\sf ML}(\mathcal{A},d)$ and H\"{o}lder's inequality, we have
\begin{eqnarray*}
   r_{\sf ML}(\mathcal{A},d) &=& \tilde{F}_\mathcal{A}\left( \bar{x}^1, \bar{x}^2, \ldots, \bar{x}^{d-2}, \{\bar{u}_j\}_{j=1}^{n_{d-1}},\{\bar{v}_k\}_{k=1}^{n_d} \right) \\
   \noalign{\medskip}
   &=& \max_{\|x^1\|_p \le 1} \tilde{F}_\mathcal{A}\left( x^1, \bar{x}^2, \ldots, \bar{x}^{d-2},\{\bar{u}_j\}_{j=1}^{n_{d-1}},\{\bar{v}_k\}_{k=1}^{n_d} \right) \\
   \noalign{\medskip}
   &=& \|w\|_q.
\end{eqnarray*}
Thus, the desired result follows from Proposition \ref{prop:prob}.
\end{proof}
\begin{thm}
For any given $d\ge3$ and $p\in(2,\infty]$, there is a randomized polynomial--time algorithm for Problem (\ref{eq:ML-relax}) that returns vectors $\hat{x}^1,\ldots,\hat{x}^{d-2},\{\hat{u}_j\}_{j=1}^{n_{d-1}},\{\hat{v}_k\}_{k=1}^{n_d}$ with the following property:
$$ \Pr\left[ \kappa^{d/2-1} \prod_{i=1}^{d-2} \sqrt{\frac{\log n_i}{n_i}} \cdot r_{\sf ML}(\mathcal{A},d) \le G(\mathcal{A},d) \le r_{\sf ML}(\mathcal{A},d) \right] \ge \frac{1}{2}. $$
Here, $G(\mathcal{A},d) = \tilde{F}_\mathcal{A}\left( \hat{x}^1,\ldots,\hat{x}^{d-2},\{\hat{u}_j\}_{j=1}^{n_{d-1}},\{\hat{v}_k\}_{k=1}^{n_d} \right)$ and
$$ \kappa = \left\{ \begin{array}{c@{\quad}l} \delta_0 & \mbox{if } p \in (2,\infty), \\
\noalign{\medskip}
\delta_1 & \mbox{if } p = \infty,
\end{array}\right.
$$
where the constants $\delta_0,\delta_1$ are given by (\ref{eq:inf-norm-const}) and (\ref{eq:q-norm-const}), respectively.
\end{thm}
\begin{proof}
We shall prove the theorem only for the case where $p\in(2,\infty)$; the case where $p=\infty$ will be similar.  The proof proceeds by induction on $d\ge3$.  The base case follows from Proposition \ref{prop:rand-diam-approx}.  Now, set $M=(2\ln2)n_1^{c_2}/c_1$, where the constants $c_1,c_2$ are given by (\ref{eq:q-norm-const}).  Consider a collection $\{\xi^i_j:i=1,\ldots,M;\, j=1,\ldots,n_1\}$ of i.i.d.~random variables with density $p\cdot\exp(-|t|^p)/(2\Gamma(1/p))$.  Define
$$ \xi^i = (\xi^i_1,\ldots,\xi^i_{n_1}) \in \R^{n_1}, \quad \bar{\xi}^i = \frac{\xi^i}{\|\xi^i\|_p} \quad\mbox{for } i=1,\ldots,M. $$
By the inductive hypothesis, there is a randomized polynomial--time algorithm that can compute, for each $i=1,\ldots,M$, a number $G(\mathcal{A}(\bar{\xi}^i),d-1)$ satisfying
$$ \Pr\left[ \delta_1^{(d-3)/2} \prod_{i=2}^{d-2} \sqrt{\frac{\log n_i}{n_i}} \cdot r_{\sf ML}(\mathcal{A}(\bar{\xi}^i),d-1) \le G(\mathcal{A}(\bar{\xi}^i),d-1) \le r_{\sf ML}(\mathcal{A}(\bar{\xi}^i),d-1) \right] \ge \frac{1}{2}. $$
Now, consider the events
$$ E_i = \left\{ G(\mathcal{A}(\bar{\xi}^i),d-1) \ge \delta_1^{(d-3)/2} \prod_{i=2}^{d-2} \sqrt{\frac{\log n_i}{n_i}} \cdot r_{\sf ML}(\mathcal{A}(\bar{\xi}^i),d-1) \right\} \quad\mbox{for } i=1,\ldots,M $$
and let
$$ \left( \bar{x}^1,\ldots,\bar{x}^{d-2},\{\bar{u}_j\}_{j=1}^{n_{d-1}},\{\bar{v}_k\}_{k=1}^{n_d} \right) $$
be an optimal solution to Problem (\ref{eq:ML-relax}).  Note that $\Pr(E_i) \ge 1/2$ for $i=1,\ldots,M$.  We compute
\begin{eqnarray}
& & \Pr\left( G(\mathcal{A}(\bar{\xi}^i),d-1) \ge \delta_1^{d/2-1} \prod_{i=1}^{d-2} \sqrt{\frac{\log n_i}{n_i}} \cdot r_{\sf ML}(\mathcal{A},d) \right) \nonumber \\
\noalign{\medskip}
&\ge& \Pr\left( G(\mathcal{A}(\bar{\xi}^i),d-1) \ge \delta_1^{d/2-1} \prod_{i=1}^{d-2} \sqrt{\frac{\log n_i}{n_i}} \cdot r_{\sf ML}(\mathcal{A},d) \,\Bigg|\, E_i \right) \times \Pr(E_i) \nonumber \\
\noalign{\medskip}
&\ge& \frac{1}{2} \Pr\left( r_{\sf ML}(\mathcal{A}(\bar{\xi}^i),d-1) \ge \sqrt{\frac{\delta_1 \log n_1}{n_1}} \cdot r_{\sf ML}(\mathcal{A},d) \right) \label{eq:prob-calc-1} \\
\noalign{\medskip}
&\ge& \frac{1}{2}\Pr\left( \tilde{F}_\mathcal{A}\left( \bar{\xi}^i,\bar{x}^2,\ldots,\bar{x}^{d-2},\{\bar{u}_j\}_{j=1}^{n_{d-1}},\{\bar{v}_k\}_{k=1}^{n_d} \right) \ge \sqrt{\frac{\delta_1 \log n_1}{n_1}} \cdot r_{\sf ML}(\mathcal{A},d) \right) \nonumber \\
\noalign{\medskip}
&\ge& \frac{c_1}{2n_1^{c_2}}, \label{eq:prob-calc-2}
\end{eqnarray}
where (\ref{eq:prob-calc-1}) follows from the fact that $\bar{\xi}^i$ is independent of the randomizations used to compute $G(\mathcal{A}(\bar{\xi}^i),d-1)$, and (\ref{eq:prob-calc-2}) follows from Proposition \ref{prop:prob-high-d}(b).  Upon setting 
$$ G(\mathcal{A},d) = \max_{1\le i\le M} G(\mathcal{A}(\bar{\xi}^i),d-1), $$
we conclude that
\begin{eqnarray*}
& & \Pr\left( G(\mathcal{A},d) \ge \delta_1^{d/2-1} \prod_{i=1}^{d-2} \sqrt{\frac{\log n_i}{n_i}} \cdot r_{\sf ML}(\mathcal{A},d) \right) \\
\noalign{\medskip}
&\ge& 1 - \prod_{i=1}^M \Pr\left(  G(\mathcal{A}(\bar{\xi}^i),d-1) < \delta_1^{d/2-1} \prod_{i=1}^{d-2} \sqrt{\frac{\log n_i}{n_i}} \cdot r_{\sf ML}(\mathcal{A},d) \right) \\
\noalign{\medskip}
&\ge& 1 - \left( 1 - \frac{c_1}{2n_1^{c_2}} \right)^M \\
\noalign{\medskip}
&\ge& \frac{1}{2}.
\end{eqnarray*}
This completes the proof.
\end{proof}
\begin{coro}
For any given $d\ge3$ and $p\in(2,\infty]$, there is a randomized polynomial--time algorithm for Problem $({\sf HP})$ with approximation ratio (resp.~relative approximation ratio) $\Omega\left( (\log n / n)^{d/2-1} \right)$ when $d$ is odd (resp.~even).
\end{coro}

\section{Conclusion} \label{sec:concl}
In this paper, we studied the hardness and approximability of homogeneous polynomial optimization and related multilinear optimization problems with $L_p$--ball constraints.  A crucial first step in our proofs is to relate the polynomial optimization problem at hand to a suitable multilinear optimization problem.  To obtain approximation results, we further showed that the $L_p$--ball constrained  multilinear optimization problem is equivalent, from an approximation perspective, to that of determining the diameters of certain convex bodies.  Such equivalence was established using the Grothendieck inequality (see, e.g.,~\cite{KN12,P12}) and an argument of Khot and Naor~\cite{KN08} (cf.~\cite{S11a}).  Consequently, by extending the approaches in~\cite{KN08,S11a} and applying results from algorithmic convex geometry, we were able to develop both deterministic and randomized polynomial--time approximation algorithms for various $L_p$--ball constrained polynomial optimization problems, whose approximation guarantees are currently the best known in the literature.  We believe that the wide array of tools used in this paper will have further applications in the study of polynomial optimization problems.  In addition, it would be interesting to find more applications of the optimization models studied in this paper.

\section*{Appendix}
\appendix

\section{Proof of Theorem \ref{hpml}} \label{app:hpml}
The proof of Theorem \ref{hpml} relies on the following polarization formula, whose proof can be found, e.g., in~\cite[Lemma 3.5]{HLZ10}:
\begin{prop}\label{polar}
Let $x^{1}, x^{2},\ldots, x^{d}\in \mathbb{R}^{n}$ be arbitrary, and let $\xi_{1}, \xi_{2},\ldots, \xi_{d}$ be i.i.d.~Bernoulli random variables (i.e., $\Pr(\xi_{i}=1)=\Pr(\xi_{i}=-1)=1/2$ for $i=1,\ldots,d$). Then, we have
\begin{equation} \label{eq:polar}
   \mathbb{E}\left[\left(\prod_{i=1}^{d}\xi_{i}\right)f_{\mathcal{A}}\left(\sum_{j=1}^{d}\xi_{j}x^{j}\right)\right] = d! \cdot F_{\mathcal{A}}( x^{1}, x^{2},\ldots, x^{d} ). 
\end{equation}
\end{prop}
Armed with Proposition \ref{polar}, we proceed as follows.  Let $(x^{1},x^{2},\ldots,x^{d})$ be the feasible solution to Problem $({\sf MR})$ returned by $\mathscr{A}_{\sf MR}$. By assumption, we have $\|x^i\|_p \le 1$ for $i=1,\ldots,d$ and $F_\mathcal{A}\left( x^1,x^2,\ldots,x^d \right) \ge \alpha v^*$.  When $d\ge3$ is odd, we can rewrite (\ref{eq:polar}) as 
$$ d! \cdot F_{\mathcal{A}}(x^{1}, x^{2},\ldots,x^{d}) = \mathbb{E}\left[\left(\prod_{i=1}^{d}\xi_{i}\right)f_{\mathcal{A}}\left(\sum_{j=1}^{d}\xi_{j}x^{j}\right)\right]=\mathbb{E}\left[f_{\mathcal{A}}\left(\sum_{j=1}^{d}\left(\prod_{i\neq j}\xi_{i}\right)x^{j}\right)\right]. $$
In particular, since $d\ge3$ is assumed to be fixed, we can find in constant time a vector $\beta=(\beta_{1},\beta_{2},\ldots,\beta_{d}) \in \{-1,1\}^d$ that satisfies
$$ f_{\mathcal{A}}\left(\sum_{j=1}^{d}\left(\prod_{i\neq j}\beta_{i}\right)x^{j}\right)\geq d! \cdot F_{\mathcal{A}}(x^{1}, x^{2},\ldots,x^{d}). $$
Now, set $\hat{x}=\sum_{j=1}^{d}\left(\prod_{i\neq j}\beta_{i}\right)x^{j} \Big/ \left\| \sum_{j=1}^{d}\left(\prod_{i\neq j}\beta_{i}\right)x^{j} \right\|_{p}$.  Then, we have $\|\hat{x}\|_p=1$; i.e., it is feasible for Problem $({\sf HP})$.  Moreover, since
$$ \left\| \sum_{j=1}^{d}\left(\prod_{i\neq j}\beta_{i}\right)x^{j} \right\|_p \le \sum_{j=1}^d \| x^j\|_p \le d, $$
we conclude that
$$ f_{\mathcal{A}}(\hat{x})\geq \frac{d!F_{\mathcal{A}}\left(x^{1}, x^{2},\ldots,x^{d}\right)}{\left\| \sum_{j=1}^{d}\left(\prod_{i\neq j}\beta_{i}\right)x^{j} \right\|_{p}^d} \geq \alpha \cdot d! \cdot d^{-d}\cdot v^* \geq \alpha\cdot d!\cdot d^{-d}\cdot \bar{v}, $$
as required.

Next, consider the case when $d\ge4$ is even.  Observe that every realization of the random vector $\xi=(\xi_1,\xi_2,\ldots,\xi_d) \in \{-1,1\}^d$ satisfies
$$ \left\| \frac{1}{d}\sum_{j=1}^{d}\xi_{j}x^{j} \right\|_p \le \frac{1}{d}\sum_{j=1}^d \| x^j \|_p \le 1; $$
i.e., $\frac{1}{d}\sum_{j=1}^{d}\xi_{j}x^{j}$ is feasible for Problem $({\sf HP})$.  Now, using the identity (\ref{eq:polar}), we compute
\begin{eqnarray*}
d! \cdot F_{\mathcal{A}}(x^{1}, x^{2},\ldots,x^{d}) &=& \mathbb{E}\left[\left(\prod_{i=1}^{d}\xi_{i}\right)f_{\mathcal{A}}\left(\sum_{j=1}^{d}\xi_{j}x^{j}\right)\right]\\
\noalign{\medskip}
&=& \frac{d^{d}}{2}\mathbb{E}\left[ f_{\mathcal{A}}\left(\frac{1}{d}\sum_{j=1}^{d}\xi_{j}x^{j}\right)-\underline{v} \;\Bigg \vert \prod_{i=1}^{d}\xi_{i}=1 \right] \\
\noalign{\medskip}
&\quad-& \frac{d^{d}}{2}\mathbb{E}\left[f_{\mathcal{A}}\left(\frac{1}{d}\sum_{j=1}^{d}\xi_{j}x^{j}\right)-\underline{v} \;\Bigg \vert \prod_{i=1}^{d}\xi_{i}=-1\right] \\
\noalign{\medskip}
&\leq& \frac{d^{d}}{2}\mathbb{E}\left[f_{\mathcal{A}}\left(\frac{1}{d}\sum_{j=1}^{d}\xi_{j}x^{j}\right)-\underline{v} \;\Bigg \vert\prod_{i=1}^{d}\xi_{i}=1\right],
\end{eqnarray*}
where the last inequality follows from the fact that $f_{\mathcal{A}}\left(\frac{1}{d}\sum_{j=1}^{d}\xi_{j}x^{j}\right)-\underline{v}$ is always non--negative.  In particular, we can find in constant time a vector $\beta=(\beta_{1},\beta_{2},\ldots,\beta_{d}) \in \{-1,1\}^d$ that satisfies $\prod_{i=1}^{d}\beta_{i}=1$ and
$$ f_\mathcal{A}\left(\frac{1}{d}\sum_{j=1}^{d}\beta_{j}x^{j}\right)-\underline{v} \ge \frac{2d!}{d^d} \cdot F_\mathcal{A}( x^1,x^2, \ldots,x^d ). $$
Upon setting $\hat{x}=\frac{1}{d}\sum_{j=1}^{d}\beta_{j}x^{j}$ and observing that $v^* \ge \bar{v} \ge \underline{v} \ge -v^*$, we obtain
$$ f_{\mathcal{A}}(\hat{x}) -\underline{v} \geq 2\alpha \cdot d! \cdot d^{-d} \cdot v^* \geq \alpha\cdot d! \cdot d^{-d} \cdot (\bar{v}-\underline{v}). $$
Moreover, we have $\|\hat{x}\|_p\le1$.  This completes the proof of Theorem \ref{hpml}. \endproof

\section{Proof of Theorem \ref{thm:MR-NPh}} \label{app:MR-NPh}
Let $d\ge3$ and $p\in[2,\infty]$ be fixed.  We shall reduce Problem $({\sf ML})$ to Problem $({\sf MR})$, again by using the symmetrization procedure introduced in Section \ref{sec:prelim}.  Towards that end, let us first establish some preparatory results.
\begin{prop} \label{prop:ml-sym-ml}
Let $\mathcal{A} \in \R^{n_1\times n_2\times \cdots \times n_d}$ be an arbitrary order--$d$ tensor and $\mbox{sym}(\mathcal{A}) \in \R^{N^d}$ be its symmetrization, where $N=n_1+n_2+\cdots+n_d$.  Moreover, let $z^i = \left[ \, (z^{i,1})^T \, (z^{i,2})^T \, \cdots \, (z^{i,d})^T \, \right]^T \in \R^N$ be given, where $z^{i,j} \in \R^{n_j}$ for $i,j=1,\ldots,d$.  Then,
$$ F_{\text{sym}(\mathcal{A})}(z^1,z^2,\ldots,z^d) = \sum_{(\pi_1,\pi_2,\ldots,\pi_d)\in S_d} F_\mathcal{A}(z^{\pi_1,1},z^{\pi_2,2},\ldots,z^{\pi_d,d}), $$
where $S_d$ is the set of permutations of $\{1,2,\ldots,d\}$.
\end{prop}
\begin{proof}
Using the sets $B_1,\ldots,B_d$ defined in (\ref{eq:partition}) and the definition of $\mbox{sym}(\mathcal{A})$, we have
\begin{eqnarray*}
F_{\text{sym}(\mathcal{A})}(z^1,\ldots,z^d) &=& \sum_{\pi=(\pi_1,\ldots,\pi_d) \in S_d} \sum_{{i_j \in B_{\pi_j}} \atop {j=1,\ldots,d}} \left[ \mbox{sym}(\mathcal{A}) \right]_{i_1\cdots i_d} z^1_{i_1}\cdots z^d_{i_d} \\
\noalign{\medskip}
&=& \sum_{\pi=(\pi_1,\ldots,\pi_d) \in S_d} \sum_{i_1=1}^{n_{\pi_1}} \cdots \sum_{i_d=1}^{n_{\pi_d}} \left[ \mathcal{A}^\pi \right]_{i_1\cdots i_d} z^{1,\pi_1}_{i_1}\cdots z^{d,\pi_d}_{i_d} \\
\noalign{\medskip}
&=& \sum_{\pi=(\pi_1,\ldots,\pi_d) \in S_d} \sum_{i_1=1}^{n_1} \cdots \sum_{i_d=1}^{n_d} \left[ \mathcal{A} \right]_{i_1\cdots i_d} z^{\pi_1^{-1},1}_{i_1}\cdots z^{\pi_d^{-1},d}_{i_d},
\end{eqnarray*}
where $\pi^{-1}=(\pi_1^{-1},\ldots,\pi_d^{-1}) \in S_d$ is the inverse of $\pi$; i.e., $\pi_{\pi_j^{-1}}=j$ for $j=1,\ldots,d$.  Consequently, we obtain
$$
F_{\text{sym}(\mathcal{A})}(z^1,\ldots,z^d) = \sum_{(\pi_1,\ldots,\pi_d) \in S_d} F_\mathcal{A}(z^{\pi_1^{-1},1},\ldots,z^{\pi_d^{-1},d}) = \sum_{(\pi_1,\ldots,\pi_d) \in S_d} F_\mathcal{A}(z^{\pi_1,1},\ldots,z^{\pi_d,d}),
$$
as desired.
\end{proof}
\begin{prop} \label{prop:aux-max}
Let $p\in[2,\infty)$ be fixed.  Given an integer $n \in [2,d]$, define the function $f_n:[0,d]^n \limto \R$ by
$$ f_n(x_1,\ldots,x_n) = \sum_{i=1}^n x_i^{1/p} \prod_{j\not=i} (d-x_j)^{1/p}. $$
Then, for any $(x_1,\ldots,x_n) \in [0,d]^n$, we have
\begin{equation} \label{eq:aux-max-opt}
f_n(x_1,\ldots,x_n) \le f_n\left( \frac{d}{n},\ldots,\frac{d}{n} \right) = d^{n/p} \cdot n^{1-1/p} \cdot \left(1-\frac{1}{n}\right)^{(n-1)/p}.
\end{equation}
\end{prop}
\begin{proof}
We prove (\ref{eq:aux-max-opt}) by induction on $n$.  For the base case (i.e., $n=2$), consider the problem:
$$ (P_2) \qquad \max\{ f_2(x_1,x_2) : 0 \le x_i \le d \mbox{ for } i=1,2 \}. $$
Note that an optimal solution to ($P_2$) must either lie on the boundary of $[0,d]^2$, or lie in the interior of $[0,d]^2$ and be a solution to the following first--order necessary conditions:
\begin{eqnarray}
x_1^{(1/p)-1}(d-x_2)^{1/p} &=& x_2^{1/p}(d-x_1)^{(1/p)-1}, \label{eq:p2-1} \\
\noalign{\medskip}
x_2^{(1/p)-1}(d-x_1)^{1/p} &=& x_1^{1/p}(d-x_2)^{(1/p)-1}. \label{eq:p2-2}
\end{eqnarray}
Consider an arbitrary $\bar{x}=(\bar{x}_1,\bar{x}_2) \in [0,d]^2$.  If $\bar{x}$ is a boundary point of $[0,d]^2$, then the structure of $f_2$ implies that
$$ f_2(\bar{x}_1,\bar{x}_2) \le \max\{ f_2(0,d),f_2(d,0) \} = d^{2/p}. $$
On the other hand, suppose that $\bar{x} \in (0,d)^2$ satisfies (\ref{eq:p2-1}) and (\ref{eq:p2-2}).  Then,
\begin{equation} \label{eq:p2-3}
\left( \frac{\bar{x}_1}{d-\bar{x}_1} \right)^{(1/p)-1} = \left( \frac{\bar{x}_2}{d-\bar{x}_2} \right)^{1/p} \quad\mbox{and}\quad \left( \frac{\bar{x}_2}{d-\bar{x}_2} \right)^{(1/p)-1} = \left( \frac{\bar{x}_1}{d-\bar{x}_1} \right)^{1/p}, 
\end{equation}
which together yield $\bar{x}_1\bar{x}_2 = (d-\bar{x}_1)(d-\bar{x}_2)$, or equivalently, $\bar{x}_1+\bar{x}_2=d$.  If $p=2$, then any $(x_1,x_2) \in [0,d]^2$ satisfying $x_1+x_2=d$ will be an optimal solution to ($P_2$).  In particular, we have $f_2(x_1,x_2) \le f_2(d/2,d/2) = d$ for all $(x_1,x_2) \in [0,d]^2$ in this case.  If $p>2$, then upon substituting $\bar{x}_1+\bar{x}_2=d$ into (\ref{eq:p2-3}), we obtain a unique solution $\bar{x}_1=\bar{x}_2=d/2$.  Since $f_2(\bar{x}_1,\bar{x}_2) = 2(d/2)^{2/p} > d^{2/p}$ for any $p>2$, we conclude that $(d/2,d/2)$ is the optimal solution to ($P_2$).  This establishes the base case.

For the inductive step, consider the problem
$$ (P_n) \qquad \max\{ f_n(x_1,\ldots,x_n) : 0 \le x_i \le d \mbox{ for } i=1,\ldots,n \}, $$
where $2<n\le d$.  Again, an optimal solution to ($P_n$) must either lie on the boundary of $[0,d]^n$, or lie in the interior of $[0,d]^n$ and be a solution to the following first--order necessary conditions:
\begin{equation} \label{eq:pn-1}
   \left( \frac{x_i}{d-x_i} \right)^{(1/p)-1} = \sum_{j\not=i} \left( \frac{x_j}{d-x_j} \right)^{1/p} \quad\mbox{for } i=1,\ldots,n. 
\end{equation}
Consider an arbitrary $\bar{x}=(\bar{x}_1,\ldots,\bar{x}_n) \in [0,d]^n$.  Suppose first that $\bar{x} \in (0,d)^n$ satisfies (\ref{eq:pn-1}).  Let $u_i = \bar{x}_i/(d-\bar{x}_i)$ for $i=1,\ldots,n$.  Then, we obtain from (\ref{eq:pn-1}) that $(u_i)^{(1/p)-1}-(u_j)^{(1/p)-1} = u_j^{1/p} - u_i^{1/p}$, or equivalently,
\begin{equation} \label{eq:pn-2}
(u_i)^{(1/p)-1}(1+u_i) = (u_j)^{(1/p)-1}(1+u_j) \quad\mbox{for } 1\le i<j\le n.
\end{equation}
It is easy to verify that the function $t\mapsto t^{(1/p)-1}(1+t)$ is strictly decreasing on $(0,p-1]$ and strictly increasing on $[p-1,\infty)$.  Thus, if we let $I_1 = \{i:u_i < p-1\}$ and $I_2 = \{i:u_i \ge p-1\}$, then (\ref{eq:pn-2}) implies that $u_i=u_j=u$ for all $i,j \in I_1$ and $u_i=u_j=v$ for all $i,j \in I_2$; i.e., $\bar{x}_i=\bar{x}_j$ whenever $i,j \in I_1$ or $i,j \in I_2$.  We claim that in fact $I_2=\emptyset$.  To prove this, let us first show that $|I_2|\le1$.  Suppose to the contrary that $|I_2|\ge2$.  Let $i,j \in I_2$ be such that $i\not=j$.  Then, from (\ref{eq:pn-1}) and the fact that $u_i>0$ for $i=1,\ldots,n$, we have
$$ u_i^{(1/p)-1} - u_j^{1/p} = \sum_{k\not=i,j} u_k^{1/p} > 0. $$
However, since $u_i,u_j \ge p-1 \ge 1$, we have $u_i^{(1/p)-1}-u_j^{1/p} \le 0$, which is a contradiction.  It follows that $|I_2|\le1$.

Now, suppose that $|I_2|=1$.  Then, from (\ref{eq:pn-1}), we have
\begin{equation} \label{eq:pn-3}
v^{(1/p)-1} = (n-1)u^{1/p}.
\end{equation}
This, together with (\ref{eq:pn-2}), implies that
$$ u^{(1/p)-1}(1+u) = (n-1)u^{1/p}\left[ 1 + (n-1)^{p/(1-p)}u^{1/(1-p)} \right], $$
or equivalently,
\begin{equation} \label{eq:pn-4}
(n-2)u + (n-1)^{1/(1-p)}u^{(p-2)/(p-1)} = 1.
\end{equation}
Since both summands in (\ref{eq:pn-4}) are non--negative, we clearly have $u\le 1/(n-2)<1$.  We claim that 
\begin{equation} \label{eq:pn-5}
u \ge \beta \equiv \frac{1-(n-1)^{1/(1-p)}}{n-2}.
\end{equation}
Indeed, suppose to the contrary that $u < \beta$.  Then, using the fact that $u<1$, we have
\begin{eqnarray*}
(n-2)u + (n-1)^{1/(1-p)}u^{(p-2)/(p-1)} &<& 1-(n-1)^{1/(1-p)} + (n-1)^{1/(1-p)}u^{(p-2)/(p-1)} \\
\noalign{\medskip}
&=& 1 + (n-1)^{1/(1-p)} \left[ u^{(p-2)/(p-1)} - 1 \right] \\
\noalign{\medskip}
&\le& 1,
\end{eqnarray*}
which contradicts (\ref{eq:pn-4}).  This establishes (\ref{eq:pn-5}).

We now show that (\ref{eq:pn-5}) leads to $v<1$, which would contradict the definition of $v$.  Indeed, using (\ref{eq:pn-5}), we have
\begin{equation} \label{eq:pn-7}
(n-1)^pu \ge \frac{(n-1)^p}{n-2} \left[ 1 - (n-1)^{1/(1-p)} \right].
\end{equation}
Consider the function $h:[2,\infty)\limto\R$ given by
$$ h(p) = (n-1)^p\left[ 1 - (n-1)^{1/(1-p)} \right]. $$
By a routine computation, we have
$$ h'(p) = \ln(n-1) \cdot (n-1)^p \cdot \left[ 1 - (n-1)^{1/(1-p)} - \frac{1}{(1-p)^2} (n-1)^{1/(1-p)} \right]. $$
Observe that
\begin{eqnarray}
& & 1 - (n-1)^{1/(1-p)} - \frac{1}{(1-p)^2} (n-1)^{1/(1-p)} \ge 0 \nonumber \\
\noalign{\medskip}
& \Longleftrightarrow & \left( \frac{1}{n-1} \right)^{1/(p-1)} \left[ 1 + \frac{1}{(p-1)^2} \right] \le 1 \nonumber \\
\noalign{\medskip}
& \Longleftarrow & 1 + \frac{1}{(p-1)^2} \le 2^{1/(p-1)}. \label{eq:pn-6}
\end{eqnarray}
It is straightforward to show that (\ref{eq:pn-6}) holds for $p\in[2,3]$ by comparing the slopes of the functions $p\mapsto 1+(p-1)^{-2}$ and $p\mapsto 2^{1/(p-1)}$.  For $p\ge3$, we have
$$ \left[ 1 + \frac{1}{(p-1)^2} \right]^{(p-1)^2} \le e < 2^2 \le 2^{p-1}, $$
which implies that (\ref{eq:pn-6}) holds.  Thus, we see that $h$ is increasing on $p \in [2,\infty)$, and from (\ref{eq:pn-7}) we obtain
$$ (n-1)^pu \ge \frac{(n-1)^p}{n-2} \left[ 1 - (n-1)^{1/(1-p)} \right] \ge \frac{(n-1)^2}{n-2} \left( 1 - \frac{1}{n-1} \right) = n-1 > 1. $$
This, together with (\ref{eq:pn-3}), implies that
$$ v = \left( (n-1)^pu \right)^{1/(1-p)} < 1, $$
which is the desired contradiction.

Thus, we have shown that $|I_2|=0$.  Using (\ref{eq:pn-1}), we then have $u^{(1/p)-1} = (n-1)u^{1/p}$, or equivalently, $u=1/(n-1)$.  It follows that $\bar{x}=(d/n,\ldots,d/n)$ is the unique solution to (\ref{eq:pn-1}).

Next, we show that if $\bar{x}$ lies on the boundary of $[0,d]^n$, then $f_n(\bar{x}_1,\ldots,\bar{x}_n) \le f_n(d/n,\ldots,d/n)$.  Towards that end, we first note that
\begin{eqnarray}
f_n\left( \frac{d}{n},\ldots,\frac{d}{n} \right) &=& d^{n/p} \cdot n^{1-1/p} \cdot \left(1-\frac{1}{n}\right)^{(n-1)/p} \nonumber \\
\noalign{\medskip}
&=& d^{n/p} \cdot n \cdot \left( \frac{1}{n-1} \right)^{1/p} \cdot \left( 1-\frac{1}{n} \right)^{n/p} \nonumber \\
\noalign{\medskip}
&\ge& d^{n/p} \cdot n \cdot \left( \frac{1}{n-1} \right)^{1/p} \cdot e^{-1/p} \cdot \left( 1-\frac{1}{n} \right)^{1/p} \nonumber \\
\noalign{\medskip}
&\ge& d^{n/p} \cdot \sqrt{n/e}, \label{eq:pn-8}
\end{eqnarray}
where the last inequality follows from the fact that $p\ge2$.  Now, suppose that $\bar{x}_i=d$ for some $i=1,\ldots,n$.  Since the function $f_n$ is symmetric in its arguments, we may assume without loss that $i=n$.  Then, using (\ref{eq:pn-8}) and the fact that $n\ge3$, we have
$$ f_n(\bar{x}_1,\ldots,\bar{x}_{n-1},d) \le d^{n/p} < f_n\left( \frac{d}{n},\ldots,\frac{d}{n} \right). $$ 
On the other hand, suppose that $\bar{x}_n=0$.  Then, by the inductive hypothesis,
\begin{eqnarray}
f_n(\bar{x}_1,\ldots,\bar{x}_{n-1},0) &=& d^{1/p} \cdot f_{n-1}(\bar{x}_1,\ldots,\bar{x}_{n-1}) \nonumber \\
\noalign{\medskip}
&\le& d^{1/p} \cdot d^{(n-1)/p} \cdot (n-1)^{1-1/p} \cdot \left(1-\frac{1}{n-1}\right)^{(n-2)/p} \nonumber \\
\noalign{\medskip}
&=& d^{n/p} (n-1) \left( \frac{1}{n-2} \right)^{1/p} \left( 1-\frac{1}{n-1} \right)^{(n-1)/p}. \label{eq:pn-9}
\end{eqnarray}
Since $n\ge3$ and $p\ge2$, we have
\begin{eqnarray*}
&& (n-1) \left( \frac{1}{n-2} \right)^{1/p} \left( 1-\frac{1}{n-1} \right)^{(n-1)/p} < n  \left( \frac{1}{n-1} \right)^{1/p}  \left( 1-\frac{1}{n} \right)^{n/p} \\
\noalign{\medskip}
&\Longleftrightarrow& \frac{n}{n-1}  \left( 1-\frac{1}{n-1} \right)^{2/p}  \left[ 1+\frac{1}{n(n-2)} \right]^{n/p} > 1 \\
\noalign{\medskip}
&\Longleftarrow& \frac{n}{n-1}  \left( 1-\frac{1}{n-1} \right)^{2/p}  \left[ 1+\frac{1}{n(n-2)} \right]^{2/p} \ge 1 \\
\noalign{\medskip}
&\Longleftrightarrow& \left( \frac{n}{n-1} \right)^{1-2/p} \ge 1.
\end{eqnarray*}
Hence, we obtain from (\ref{eq:pn-9}) that
$$ f_n(\bar{x}_1,\ldots,\bar{x}_{n-1},0) < d^{n/p} \cdot n \cdot \left( \frac{1}{n-1} \right)^{1/p} \cdot \left( 1-\frac{1}{n} \right)^{n/p} = f_n\left( \frac{d}{n},\ldots,\frac{d}{n} \right). $$
This completes the proof of Proposition \ref{prop:aux-max}.
\end{proof}
\begin{prop} \label{prop:ml-symml-eqv}
Let $\mathcal{A} \in \R^{n_1\times n_2\times \cdots \times n_d}$ be an arbitrary order--$d$ tensor and $\mbox{sym}(\mathcal{A}) \in \R^{N^d}$ be its symmetrization, where $N=n_1+n_2+\cdots+n_d$.  Consider the optimization problems
\begin{equation} \tag{$A_d$}
\begin{array}{ccc@{\quad}l}
\tau(A_d) &=& \mbox{maximize} & d! \cdot F_\mathcal{A}(x^1,x^2,\ldots,x^d) \\
\noalign{\medskip}
& & \mbox{subject to} & \|x^i\|_p \le 1, \, x^i \in \R^{n_i} \quad\mbox{for } i=1,\ldots,d
\end{array}
\end{equation}
and
\begin{equation} \tag{$B_d$}
\begin{array}{ccc@{\quad}l}
\tau(B_d) &=& \mbox{maximize} & F_{\text{sym}(\mathcal{A})}(z^1,z^2,\ldots,z^d) \\
\noalign{\medskip}
& & \mbox{subject to} & \|z^i\|_p \le d^{1/p}, \, z^i \in \R^N \quad\mbox{for } i=1,\ldots,d,
\end{array}
\end{equation}
where $F_\mathcal{A}$ (resp.~$F_{\text{sym}}(\mathcal{A})$) is the multilinear form associated with $\mathcal{A}$ (resp.~$\mbox{sym}(\mathcal{A})$).  Then, the following hold:
\begin{enumerate}
   \item[\subpb] $\tau(A_d)=\tau(B_d)$.

   \item[\subpb]  Let $(\bar{x}^1,\bar{x}^2,\ldots,\bar{x}^d) \in \R^{n_1} \times \R^{n_2} \times \cdots \times \R^{n_d}$ be an optimal solution to Problem $(A_d)$.  Set
\begin{equation} \label{eq:hat-soln}
\hat{z}^i = \left[ \, (\bar{x}^1)^T \, (\bar{x}^2)^T \, \cdots \, (\bar{x}^d)^T \, \right]^T \in \R^N \quad \mbox{for } i=1,\ldots,d.
\end{equation}
Then, $(\hat{z}^1,\hat{z}^2,\ldots,\hat{z}^d)$ constitutes an optimal solution to Problem $(B_d)$.
   
   \item[\subpb] Let $(\tilde{z}^1,\tilde{z}^2,\ldots,\tilde{z}^d) \in \R^N \times \R^N \times \cdots \times \R^N$ be an optimal solution to Problem $(B_d)$ with
\begin{equation} \label{eq:tilde-soln}
\tilde{z}^i = \left[ \, (\tilde{z}^{i,1})^T \, \cdots \, (\tilde{z}^{i,d})^T \, \right]^T \in \R^N, \quad \tilde{z}^{i,j} \in \R^{n_j} \quad\mbox{for } i,j=1,\ldots,d.
\end{equation}
Then, $\|\tilde{z}^{i,j}\|_p = 1$ for $i,j=1,\ldots,d$.  Moreover, there exists a vector $\bar{z} \in \R^N$ with
$$ \bar{z} = \left[ \, (\hat{x}^1)^T \, (\hat{x}^2)^T \, \cdots \, (\hat{x}^d)^T \, \right]^T, \quad \hat{x}^i \in \R^{n_i} \quad\mbox{for } i=1,\ldots,d, $$
such that $(\bar{z},\ldots,\bar{z})$ is an optimal solution to Problem $(B_d)$ and $(\hat{x}^1,\hat{x}^2,\ldots,\hat{x}^d)$ is an optimal solution to Problem $(A_d)$.
\end{enumerate}
\resetspb
\end{prop}
\begin{proof}
By Proposition \ref{prop:ml-sym-ml}, Problem ($B_d$) is equivalent to
\begin{equation} \tag{$B_d'$}
\begin{array}{ccc@{\quad}l}
\tau(B_d) &=& \mbox{maximize} & \displaystyle{ \sum_{\pi \in S_d}  F_\mathcal{A}(z^{\pi_1,1},\ldots,z^{\pi_d,d}) } \\
\noalign{\medskip}
& & \mbox{subject to} & \|z^i\|_p \le d^{1/p}, \, z^i \in \R^N \quad\mbox{for } i=1,\ldots,d.
\end{array}
\end{equation}
If $(\bar{x}^1,\ldots,\bar{x}^d) \in \R^{n_1} \times \cdots \times \R^{n_d}$ is an optimal solution to Problem ($A_d$), then the solution $(\hat{z}^1,\ldots,\hat{z}^d) \in \R^N \times \cdots \times \R^N$ as defined in (\ref{eq:hat-soln}) is feasible for Problem $(B_d)$.  Moreover, we have
$$ \tau(B_d) \ge \sum_{\pi \in S_d}  F_\mathcal{A}(\hat{z}^{\pi_1,1},\ldots,\hat{z}^{\pi_d,d}) = d!\cdot F_\mathcal{A}(\bar{x}^1,\ldots,\bar{x}^d) = \tau(A_d). $$
Hence, (b) is implied by (a).

To prove (a) and (c), we consider two cases:

\medskip
\noindent {\bf Case 1: $p=\infty$.} Consider an optimal solution $(\tilde{z}^1,\ldots,\tilde{z}^d) \in \R^N \times \cdots \times \R^N$ to Problem $(B_d)$ with $\tilde{z}^i$ given by (\ref{eq:tilde-soln}).  Then, $(\tilde{z}^1,\ldots,\tilde{z}^d)$ is also optimal for Problem $(B_d')$.  Let $\tau \in S_d$ be a permutation of $\{1,\ldots,d\}$ satisfying
\begin{equation} \label{eq:max-permut}
F_\mathcal{A}(\tilde{z}^{\tau_1,1},\ldots,\tilde{z}^{\tau_d,d}) \ge F_\mathcal{A}(\tilde{z}^{\pi_1,1},\ldots,\tilde{z}^{\pi_d,d}) \quad\mbox{for all } \pi \in S_d.
\end{equation}
By the optimality of $(\tilde{z}^1,\ldots,\tilde{z}^d)$ for Problem $(B_d')$ and the fact that $\|\tilde{z}^i\|_\infty  = \max_{1\le j \le d} \|\tilde{z}^{i,j}\|_\infty$, we have $\|\tilde{z}^{i,j}\|_\infty=1$ for $i,j=1,\ldots,d$.  Now, set $\hat{x}^i = \tilde{z}^{\tau_i,i}$ for $i=1,\ldots,d$ and form $\bar{z} = \left[ \, (\hat{x}^1)^T \, \cdots \, (\hat{x}^d)^T \, \right]^T \in \R^N$.  By construction, we have $\|\bar{z}\|_\infty=1$ and hence $(\bar{z},\ldots,\bar{z})$ is feasible for Problem $(B_d')$.  Using (\ref{eq:max-permut}), we compute
\begin{equation} \label{eq:inf-ineq-1}
\sum_{\pi \in S_d} F_\mathcal{A}(\hat{x}^1,\ldots,\hat{x}^d) = \sum_{\pi \in S_d} F_\mathcal{A}(\tilde{z}^{\tau_1,1},\ldots,\tilde{z}^{\tau_d,d}) \ge \sum_{\pi \in S_d} F_\mathcal{A}(\tilde{z}^{\pi_1,1},\ldots,\tilde{z}^{\pi_d,d}) = \tau(B_d),
\end{equation}
which certifies the optimality of $(\bar{z},\ldots,\bar{z})$ for Problem $(B_d')$ and hence also for Problem $(B_d)$.   Moreover, since $\|\hat{x}^i\|_\infty \le \|\bar{z}\|_\infty = 1$ for $i=1,\ldots,d$, $(\hat{x}^1,\ldots,\hat{x}^d)$ is feasible for Problem $(A_d)$.  This implies that
\begin{equation} \label{eq:inf-ineq-2}
\sum_{\pi \in S_d} F_\mathcal{A}(\hat{x}^1,\ldots,\hat{x}^d) = d! \cdot F_\mathcal{A}(\hat{x}^1,\ldots,\hat{x}^d) \le \tau(A_d).
\end{equation}
Upon combining (\ref{eq:inf-ineq-1}) and (\ref{eq:inf-ineq-2}), we have $\tau(A_d) \ge \tau(B_d)$, and that $(\hat{x}^1,\ldots,\hat{x}^d)$ is an optimal solution to Problem $(A_d)$.  This establishes (a) and (c) for this case.

\medskip
\noindent{\bf Case 2: $p \in [2,\infty)$.} We prove (a) and (c) by induction on $d\ge2$.  For the base case (i.e., $d=2$), we have $\mathcal{A} \in \R^{n_1 \times n_2}$.  Hence, we can write
\begin{equation} \tag{$A_2$}
\tau(A_2) = 2\max\left\{ x^T\mathcal{A}y : \|x\|_p \le 1, \, \|y\|_p \le 1 \right\} 
\end{equation}
and
\begin{equation} \tag{$B_2$}
\tau(B_2) = \max\left\{ (v^1)^T\mathcal{A}w^2 + (w^1)^T\mathcal{A}v^2 : \|v^1\|_p^p + \|v^2\|_p^p \le 2, \, \|w^1\|_p^p + \|w^2\|_p^p \le 2 \right\}.
\end{equation}
Let $(\tilde{z}^1,\tilde{z}^2) \in \R^{n_1+n_2} \times \R^{n_1+n_2}$, where $\tilde{z}^1 = \left[ \, (\tilde{v}^1)^T \, (\tilde{v}^2)^T \,\right]$ and $\tilde{z}^2 = \left[ \, (\tilde{w}^1)^T \, (\tilde{w}^2)^T \,\right]$, be an optimal solution to Problem $(B_2)$.
Suppose that $\|\tilde{v}^1\|_p^p=k_1$ and $\|\tilde{w}^1\|_p^p=k_2$.  Then,
\begin{eqnarray*}
\tau(B_2) &\le& \max\left\{ (v^1)^T\mathcal{A}w^2 : \|v^1\|_p^p \le k_1, \, \|w^2\|_p^p \le 2-k_2 \right\} \\
\noalign{\medskip}
&\quad+& \max\left\{ (w^1)^T\mathcal{A}v^2 : \|w^1\|_p^p \le k_2, \, \|v^2\|_p^p \le 2-k_1 \right\} \\
\noalign{\medskip}
&=& \frac{\tau(A_2)}{2} \left[ k_1^{1/p}(2-k_2)^{1/p} + k_2^{1/p}(2-k_1)^{1/p} \right] \\
\noalign{\medskip}
&\le& \tau(A_2), 
\end{eqnarray*}
where the last inequality follows from Proposition \ref{prop:aux-max} and the fact that $0\le k_1,k_2\le 2$.  This establishes (a).  Moreover, since all the above inequalities hold as equalities, from the proof of Proposition \ref{prop:aux-max}, both $k_1,k_2$ must equal to $1$ when $p>2$ and can be taken as $1$ when $p=2$.  This, together with the optimality of $(\tilde{z}^1,\tilde{z}^2)$, implies that we can take $\tilde{v}^1=\tilde{w}^1$ and $\tilde{v}^2=\tilde{w}^2$.  Upon setting $\hat{x}^i = \tilde{v}^i$ for $i=1,2$ and forming $\bar{z} = \left[ \, (\hat{x}^1)^T \, (\hat{x}^2)^T \, \right]$, it can be verified that (c) holds.  Thus, the base case is established.

Next, consider an optimal solution $(\tilde{z}^1,\ldots,\tilde{z}^d)$ to Problem $(B_d)$ with $\tilde{z}^i$ given by (\ref{eq:tilde-soln}).  Suppose that $\|\tilde{z}^{i,d}\|_p^p=k_i$ for $i=1,\ldots,d$.  Then,
\begin{eqnarray}
\tau(B_d) &=& \sum_{\pi \in S_d} F_\mathcal{A}(\tilde{z}^{\pi_1,1},\ldots,\tilde{z}^{\pi_d,d}) \nonumber \\
\noalign{\medskip}
&=& \sum_{i=1}^d \sum_{\pi \in S_d:\pi_d=i} F_{\mathcal{A}(\tilde{z}^{i,d})}(\tilde{z}^{\pi_1,1},\ldots,\tilde{z}^{\pi_{d-1},d-1}) \nonumber \\
\noalign{\medskip}
&\le& \sum_{i=1}^d \max_{{\|w^j\|_p^p \le d-k_j, \, w^j \in \R^{N-n_d}}\atop{j=1,\ldots,i-1,i+1,\ldots,d}} \sum_{\pi \in S_d:\pi_d=i} F_{\mathcal{A}(\tilde{z}^{i,d})}(w^{\pi_1,1},\ldots,w^{\pi_{d-1},d-1}) \nonumber \\
\noalign{\medskip}
&=& \sum_{i=1}^d \prod_{j\not=i} \left( \frac{d-k_j}{d-1} \right)^{1/p} \max_{{\|w^j\|_p^p \le d-1, \, w^j \in \R^{N-n_d}}\atop{j=1,\ldots,d-1}} \sum_{\pi \in S_{d-1}} F_{\mathcal{A}(\tilde{z}^{i,d})}(w^{\pi_1,1},\ldots,w^{\pi_{d-1},d-1}) \nonumber \\
\noalign{\medskip}
&=& (d-1)! \sum_{i=1}^d \prod_{j\not=i} \left( \frac{d-k_j}{d-1} \right)^{1/p} \max_{{\|x^j\|_p \le 1, \, x^j \in \R^{n_j}} \atop {j=1,\ldots,d-1}}  F_{\mathcal{A}(\tilde{z}^{i,d})}(x^1,\ldots,x^{d-1}) \label{eq:eqv-2} \\
\noalign{\medskip}
&\le& (d-1)! \sum_{i=1}^d \prod_{j\not=i} \left( \frac{d-k_j}{d-1} \right)^{1/p} \max_{{\|x^j\|_p \le 1, \, x^j \in \R^{n_j}} \atop {j=1,\ldots,d-1}}  \max_{\|x^d\|_p^p \le k_i, \, x^d \in \R^{n_d}} F_\mathcal{A}(x^1,\ldots,x^{d-1},x^d) \nonumber 
\end{eqnarray}
\begin{eqnarray}
&=& \frac{\tau(A_d)}{d(d-1)^{(n-1)/p}} \sum_{i=1}^d k_i^{1/p} \prod_{j\not=i} (d-k_j)^{1/p} \qquad\qquad\qquad\qquad\qquad\qquad\qquad \nonumber \\
\noalign{\medskip}
&\le& \tau(A_d), \label{eq:eqv-3}
\end{eqnarray}
where (\ref{eq:eqv-2}) follows from Proposition \ref{prop:ml-sym-ml} and the inductive hypothesis, and (\ref{eq:eqv-3}) follows from Proposition \ref{prop:aux-max} and the fact that $0\le k_i \le d$ for $i=1,\ldots,d$.  This establishes (a).  Moreover, since all the above inequalities hold as equalities, the proof of Proposition \ref{prop:aux-max} shows that we must have $k_i=1$ for $i=1,\ldots,d$.  This implies that $\|\tilde{z}^{i,d}\|_p = 1$ for $i=1,\ldots,d$.  By repeating the above argument using the group $\{\tilde{z}^{i,j}:i=1,\ldots,d\}$ in place of $\{\tilde{z}^{i,d}:i=1,\ldots,d\}$ for each $j=1,\ldots,d-1$, we conclude that $\|\tilde{z}^{i,j}\|_p = 1$ for $i,j=1,\ldots,d$.  Now, as in Case 1, let $\tau \in S_d$ be a permutation of $\{1,\ldots,d\}$ satisfying (\ref{eq:max-permut}).  Set $\hat{x}^i = \tilde{z}^{\tau_i,i}$ for $i=1,\ldots,d$ and form $\bar{z} = \left[ \, (\hat{x}^1)^T \, \cdots \, (\hat{x}^d)^T \, \right]^T \in \R^N$.  By construction, we have 
$$ \|\bar{z}\|_p^p = \sum_{i=1}^d \|\hat{x}^i\|_p^p = \sum_{i=1}^d \|\tilde{z}^{\tau_i,i}\|_p^p = d $$
and hence $(\bar{z},\ldots,\bar{z})$ is feasible for Problem $(B_d)$.  It remains to argue as in Case 1 to complete the inductive step and also the proof of Proposition \ref{prop:ml-symml-eqv}.
\end{proof}

\medskip
Proposition \ref{prop:ml-symml-eqv} implies that for any given $d\ge3$ and $p\in[2,\infty]$, any instance of Problem $({\sf ML})$ can be converted into an instance of Problem $({\sf MR})$ in polynomial time.  Since Problem $({\sf ML})$ is NP--hard by Proposition \ref{prop:ML-NPh}, it follows that Problem $({\sf MR})$ is also NP--hard.  This completes the proof of Theorem \ref{thm:MR-NPh}. \endproof


\bibliography{sdpbib}

\begin{thebibliography}{10}

\bibitem{AOPT11}
A.~A. Ahmadi, A.~Olshevsky, P.~A. Parrilo, and J.~N. Tsitsiklis.
\newblock {NP--Hardness of Deciding Convexity of Quartic Polynomials and
  Related Problems}.
\newblock Accepted for publication in {\it Mathematical Programming}, 2011.

\bibitem{AN06}
N.~Alon and A.~Naor.
\newblock Approximating the {Cut--Norm} via {G}rothendieck's {I}nequality.
\newblock {\em SIAM Journal on Computing}, 35(4):787--803, 2006.

\bibitem{BBP98}
L.~Baratchart, M.~Berthod, and L.~Pottier.
\newblock {Optimization of Positive Generalized Polynomials under $\ell^p$
  Constraints}.
\newblock {\em Journal of Convex Analysis}, 5(2):353--379, 1998.

\bibitem{Barvinok07}
A.~Barvinok.
\newblock Integration and {O}ptimization of {M}ultivariate {P}olynomials by
  {R}estriction onto a {R}andom {S}ubspace.
\newblock {\em Foundations of Computational Mathematics}, 7(2):229--244, 2007.

\bibitem{BN01b}
A.~Ben-Tal and A.~Nemirovski.
\newblock On {A}pproximating {M}atrix {N}orms.
\newblock Manuscript, 2001.

\bibitem{BV11}
A.~Bhaskara and A.~Vijayaraghavan.
\newblock {Approximating Matrix $p$--Norms}.
\newblock In {\em Proceedings of the 22nd Annual ACM--SIAM Symposium on
  Discrete Algorithms (SODA 2011)}, pages 497--511, 2011.

\bibitem{BMMN11}
M.~Braverman, K.~Makarychev, Y.~Makarychev, and A.~Naor.
\newblock {The Grothendieck Constant is Strictly Smaller than Krivine's Bound}.
\newblock In {\em Proceedings of the 52nd Annual IEEE Symposium on Foundations
  of Computer Science (FOCS 2011)}, pages 453--462, 2011.

\bibitem{BGK+01}
A.~Brieden, P.~Gritzmann, R.~Kannan, V.~Klee, L.~Lov\'{a}sz, and M.~Simonovits.
\newblock Deterministic and {R}andomized {Polynomial--Time} {A}pproximation of
  {R}adii.
\newblock {\em Mathematika}, 48:63--105, 2001.

\bibitem{dKLP06}
E.~de~Klerk, M.~Laurent, and P.~A. Parrilo.
\newblock A {PTAS} for the {M}inimization of {P}olynomials of {F}ixed {D}egree
  over the {S}implex.
\newblock {\em Theoretical Computer Science}, 361(2--3):210--225, 2006.

\bibitem{GLS93}
M.~Gr\"{o}tschel, L.~Lov\'{a}sz, and A.~Schrijver.
\newblock {\em Geometric Algorithms and Combinatorial Optimization}, volume~2
  of {\em Algorithms and Combinatorics}.
\newblock Springer--Verlag, Berlin Heidelberg, second corrected edition, 1993.

\bibitem{HLZ10}
S.~He, Z.~Li, and S.~Zhang.
\newblock Approximation {A}lgorithms for {H}omogeneous {P}olynomial
  {O}ptimization with {Q}uadratic {C}onstraints.
\newblock {\em Mathematical Programming, Series B}, 125(2):353--383, 2010.

\bibitem{HL09}
C.~J. Hillar and L.-H. Lim.
\newblock Most {T}ensor {P}roblems are {NP} {H}ard.
\newblock Preprint, 2009.

\bibitem{KN08}
S.~Khot and A.~Naor.
\newblock Linear {E}quations {M}odulo $2$ and the {$L_1$} {D}iameter of
  {C}onvex {B}odies.
\newblock {\em SIAM Journal on Computing}, 38(4):1448--1463, 2008.

\bibitem{KN12}
S.~Khot and A.~Naor.
\newblock {Grothendieck--Type Inequalties in Combinatorial Optimization}.
\newblock {\em Communications on Pure and Applied Mathematics},
  65(7):992--1035, 2012.

\bibitem{KNS10}
G.~Kindler, A.~Naor, and G.~Schechtman.
\newblock The {UGC} {H}ardness {T}hreshold of the {$L_p$} {G}rothendieck
  {P}roblem.
\newblock {\em Mathematics of Operations Research}, 35(2):267--283, 2010.

\bibitem{LHZ12}
Z.~Li, S.~He, and S.~Zhang.
\newblock {\em {Approximation Methods for Polynomial Optimization}}.
\newblock SpringerBriefs in Optimization. Springer Science+Business Media, LLC,
  New York, 2012.

\bibitem{L05}
L.-H. Lim.
\newblock Singular {V}alues and {E}igenvalues of {T}ensors: {A} {V}ariational
  {A}pproach.
\newblock In {\em Proceedings of the 1st IEEE International Workshop on
  Computational Advances in Multi--Sensor Adaptive Processing (CAMSAP 2005)},
  pages 129--132, 2005.

\bibitem{LNQY09}
C.~Ling, J.~Nie, L.~Qi, and Y.~Ye.
\newblock Biquadratic {O}ptimization over {U}nit {S}pheres and {S}emidefinite
  {P}rogramming {R}elaxations.
\newblock {\em SIAM Journal on Optimization}, 20(3):1286--1310, 2009.

\bibitem{LZQ12}
C.~Ling, X.~Zhang, and L.~Qi.
\newblock {Semidefinite Relaxation Approximation for Multivariate Bi--Quadratic
  Optimization with Quadratic Constraints}.
\newblock {\em Numerical Linear Algebra with Applications}, 19(1):113--131,
  2012.

\bibitem{LZ10}
Z.-Q. Luo and S.~Zhang.
\newblock A {S}emidefinite {R}elaxation {S}cheme for {M}ultivariate {Q}uartic
  {P}olynomial {O}ptimization with {Q}uadratic {C}onstraints.
\newblock {\em SIAM Journal on Optimization}, 20(4):1716--1736, 2010.

\bibitem{N00}
{\relax Yu}.~Nesterov.
\newblock Global {Q}uadratic {O}ptimization via {C}onic {R}elaxation.
\newblock In H.~Wolkowicz, R.~Saigal, and L.~Vandenberghe, editors, {\em
  Handbook of Semidefinite Programming: Theory, Algorithms, and Applications},
  volume~27 of {\em International Series in Operations Research and Management
  Science}, pages 363--387. Kluwer Academic Publishers, Boston, Massachusetts,
  2000.

\bibitem{N03}
{\relax Yu}.~Nesterov.
\newblock Random {W}alk in a {S}implex and {Q}uadratic {O}ptimization over
  {C}onvex {P}olytopes.
\newblock {CORE} {D}iscussion {P}aper 2003071, Universit\'{e} Catholique de
  Louvain, Belgium, 2003.

\bibitem{P12}
G.~Pisier.
\newblock {Grothendieck's Theorem, Past and Present}.
\newblock {\em Bulletin (New Series) of the American Mathematical Society},
  49(2):237--323, 2012.

\bibitem{Q05}
L.~Qi.
\newblock Eigenvalues of a {R}eal {S}upersymmetric {T}ensor.
\newblock {\em Journal of Symbolic Computation}, 40(6):1302--1324, 2005.

\bibitem{RvL11}
S.~Ragnarsson and C.~F. {Van Loan}.
\newblock {Block Tensors and Symmetric Embeddings}.
\newblock Accepted for publication in {\it Linear Algebra and Its
  Applications}, 2011.

\bibitem{S11a}
A.~M.-C. So.
\newblock Deterministic {A}pproximation {A}lgorithms for {S}phere {C}onstrained
  {H}omogeneous {P}olynomial {O}ptimization {P}roblems.
\newblock {\em Mathematical Programming, Series B}, 129(2):357--382, 2011.

\bibitem{SZY07}
A.~M.-C. So, J.~Zhang, and Y.~Ye.
\newblock On {A}pproximating {C}omplex {Q}uadratic {O}ptimization {P}roblems
  via {S}emidefinite {P}rogramming {R}elaxations.
\newblock {\em Mathematical Programming, Series B}, 110(1):93--110, 2007.

\bibitem{Steinberg05}
D.~Steinberg.
\newblock Computation of {M}atrix {N}orms with {A}pplications to {R}obust
  {O}ptimization.
\newblock Master's thesis, Technion---Israel Institute of Technology, Technion
  City, Haifa 32000, Israel, 2005.

\bibitem{YY12}
Y.~Yang and Q.~Yang.
\newblock {On Solving Biquadratic Optimization via Semidefinite Relaxation}.
\newblock Accepted for publication in {\it Computational Optimization and
  Applications}, 2012.

\bibitem{ZLQ11}
X.~Zhang, C.~Ling, and L.~Qi.
\newblock {Semidefinite Relaxation Bounds for Bi--Quadratic Optimization
  Problems with Quadratic Constraints}.
\newblock {\em Journal of Global Optimization}, 49(2):293--311, 2011.

\bibitem{ZQY12}
X.~Zhang, L.~Qi, and Y.~Ye.
\newblock {The Cubic Spherical Optimization Problems}.
\newblock {\em Mathematics of Computation}, 81(279):1513--1525, 2012.

\end{thebibliography}
\bibliographystyle{abbrv}
\end{document}